\documentclass[11pt]{amsart}
\usepackage{amssymb,amsthm,verbatim, amsmath,latexsym,epsfig, amscd, graphicx, graphics, mathrsfs}
\usepackage[dvips]{color}

\newtheorem{theorem}{Theorem}[section]
\newtheorem{question}{Question}[section]
\newtheorem{lemma}[theorem]{Lemma}
\newtheorem{corollary}[theorem]{Corollary}

\theoremstyle{definition}
\newtheorem{definition}[theorem]{Definition}

\theoremstyle{remark}
\newtheorem{remark}[theorem]{Remark}
\numberwithin{equation}{section}

\newcommand{\m}{\mathrm{mod}}

\newcommand{\g}{\gamma}

\newcommand{\eps}{\varepsilon}
\newcommand{\la}{\lambda}

\newcommand{\G}{\Gamma}

\newcommand{\diam}{\mathrm{diam}}
\newcommand{\dist}{\mathrm{dist}}

\newcommand{\calA}{\mathcal{A}}
\newcommand{\calE}{\mathcal{E}}
\newcommand{\calB}{\mathcal{B}}

\newcommand{\calH}{\mathcal{H}}
\newcommand{\calL}{\mathcal{L}}

\def\Bbb{\mathbb}

\def\dist{{\rm{dist}}}

\def\dist{{\rm{dist}}}
\def\diam{{\rm{diam}}}
\def\dim{{\rm{dim}}}

\def\Bbb{\mathbb}
\def\reals{\Bbb R}

\def\cal{\mathcal}
\renewcommand{\mod}{\mathrm{mod}}
\DeclareMathOperator*{\esssup}{ess\,sup}
\begin{document}
\title[QS distortion of Ahlfors regular subsets]{Quasisymmetric dimension distortion \\of
Ahlfors regular subsets of a metric space}
\author[C.J.Bishop, H.Hakobyan and M. Williams]{Christopher J. Bishop, Hrant Hakobyan and Marshall Williams}
\address{Department of Mathematics, Stony Brook University, Stony Brook, NY 11794-3651.}
\thanks{C.~B. was partially supported by NSF Grant DMS-13-05233.}
\email{bishop@math.sunysb.edu}
%
\address{Department of Mathematics, Kansas State University, Manhattan, KS, 66506-2602}
\thanks{H.~H. was partially supported by Kansas NSF EPSCoR Grant NSF68311}
\email{hakobyan@math.ksu.edu}
%
\address{Department of Mathematics, Kansas State University, Manhattan, KS, 66506-2602}
\email{mcwill@math.ksu.edu}

\begin{abstract}
We show that if $f:X\to Y$ is a quasisymmetric mapping between Ahlfors
regular spaces, then $\dim_H f(E)\leq\dim_H E$ for ``almost every'' bounded
Ahlfors regular set $E\subseteq X$. If additionally, $X$ and $Y$ are
Loewner spaces then $\dim_H f(E)=\dim_H E$ for ``almost every" Ahlfors
regular set $E\subset X$. The precise statements of these results are given
in terms of Fuglede's modulus of measures. As a corollary of these general
theorems we show that if $f$ is a quasiconformal map of $\mathbb{R}^N$,
$N\geq 2$, then for Lebesgue a.e. $y\in\mathbb{R}^N$ we have $\dim_H f(y+E)
= \dim_H E$. A similar result holds for Carnot groups as well.

 For planar quasiconformal maps, our general estimates imply that if $E
\subset \reals$ is Ahlfors $d$-regular, $d<1$, then some component of $f(E
\times \reals)$ has dimension at most $2/(d+1)$, and we construct examples
to show this bound is sharp. In addition, we show there is a
$1$-dimensional set $S\subseteq \mathbb R$ and planar quasiconformal map
$f$ such that $f(\reals \times S)$ contains no rectifiable sub-arcs. These
results generalize work of Balogh, Monti and Tyson \cite{Tyson:frequency}
and answer questions posed in \cite{Tyson:frequency} and \cite{AimPL}.
\end{abstract}

\date{\today}
\maketitle
\tableofcontents
\newpage

\section{Introduction}
\subsection{Dimension preservation of random translates}
A quasiconformal image of a single line can be a fractal curve of Hausdorff
dimension greater than $1$. However, a quasiconformal mapping $f$ of
$\mathbb{R}^N$ is ACL for $N\geq 2$, cf. \cite{Ahlfors:QClectures,
Vaisala:lectures}. This means that $f$ is absolutely continuous on almost all
lines parallel to coordinate axes. It follows, that the image of almost every
such line is a locally rectifiable curve of Hausdorff dimension $1$. In
particular, for a quasiconformal mapping $f$ of $\mathbb{R}^N$ and a line segment
$L\subset\mathbb{R}^N$ we have
\begin{equation}\label{intervaldimpreservation}
\dim_H f(y+L)= \dim_H L = 1,
\end{equation}
for $\calH^N$-a.e. $y\in\mathbb{R}^N$. In this paper we prove the following
generalization of this result.

\begin{theorem}\label{thm:randomtranslates}
If  $f$ is a quasiconfomal mapping of $\mathbb{R}^N$, $N\geq 2,$ and
$E\subset\mathbb{R}^N$ is a bounded Ahlfors regular set then
  \begin{equation}
      \dim_H f(y + E) = \dim_H E,
  \end{equation}
for $\calH^N$-a.e. $y\in\mathbb{R}^N$.
\end{theorem}

Here $E$ is an Ahlfors $d$-regular set if Hausdorff $d$-measure of a ball of
radius $r$ centered at a point in $E$ is comparable to $r^d$, see Section
\ref{Section:preliminaries} for the precise definition.


Theorem \ref{thm:randomtranslates} is a special case of Theorem
\ref{thm:carnot} on Carnot groups. The latter says that for a Carnot group
$\mathbb{G}$ of Hausdorff dimension $Q$ (with respect to its
Carnot-Caratheodory metric)  we have
$$\dim_H f(y \cdot E) = \dim_H E,$$
for  $\calH^Q$-a.e.  $y\in\mathbb{G}$ and for
 any  bounded Ahlfors regular  subset $E\subset\mathbb{G}$.
Theorem \ref{thm:carnot} is stated and proved in Section
\ref{Section:translatesproductsproofs} by  using  Theorem \ref{thm:equaldim}
(see the next paragraph)  and an estimate for modulus of measures, Lemma
\ref{lem:translatesmodulus}. Theorem \ref{thm:equaldim}, in turn, is a
special case of a more general result about the modulus of the quasisymmetric
image of  a family of lower-regular measures  with respect to an
upper-regular base measure  (see Theorem \ref{theorem:expansion}).



\subsection{Dimension distortion of ``generic" Ahlfors regular subsets}

Almost as well known as the ACL property is the slightly stronger fact that a
quasiconformal mapping of $\mathbb{R}^{N}$ is absolutely continuous on
``almost every'' curve, where ``almost every'' is understood in the sense of
conformal modulus of curve families (see Section \ref{section:moduli} below).
A fortiori, if $f$ is a quasisymmetric mapping of $\mathbb{R}^N$, we have
\begin{equation}\label{dimequality:curves}
\dim_H f(\g) = \dim_H \g=1
\end{equation}
for ``almost every" rectifiable curve $\g$ in $\mathbb{R}^N$.

In this paper we utilize Fuglede's modulus of families of measures \cite{Fug}
to introduce the notion of ``almost every Ahlfors $d$-regular subset" of a
metric measure space $X$, see Section \ref{section:conformal} below. This in
turn allows us to generalize equality (\ref{dimequality:curves}) and to show
that dimension preservation of ``generic subsets" holds under a mild
assumption.

\begin{theorem}\label{thm:equaldim}
If $f:X\to Y$ is a quasisymmetric mapping between Ahlfors $D$-regular spaces,
$D>1$, which satisfies condition $N^{-1}$, then for every $0 < d \leq D$ we
have
\begin{equation}
  \dim_H f(E) = \dim_H E,
\end{equation}
for $\mod_{D/d}$-almost every bounded Ahlfors $d$-regular set $E\subset X$.
\end{theorem}



Theorem \ref{thm:equaldim} is proven in Section
\ref{section:nonexpansionproofs}. Recall that a homeomorphism $f\colon
X\rightarrow Y$ between Ahlfors $D$-regular metric spaces satisfies Lusin's
condition $N$ if whenever $A\subseteq X$ is such that $\calH^D(A)=0$ then
$\calH^D(f(A))=0$. We say that $f$ satisfies condition $N^{-1}$ if $f^{-1}$
satisfies condition $N$.

We say that condition $N^{-1}$ is mild because it is known to hold, not only
in the classical case where $X$ and $Y$ are Euclidean domains, but more
generally, when $X$ and $Y$ are domains in $D$-regular, $D$-Loewner spaces,
see e.g. \cite{HKST}. Recall, that Loewner spaces constitute a large class of
metric spaces introduced by Heinonen and Koskela \cite{HK:Acta}, that
includes Carnot groups equipped with their Carnot-Carath\'eodory metrics.
Indeed, conditions $N$ and $N^{-1}$ are always satisfied for quasiconformal
mappings in this setting. See \cite{HK:Acta} and \cite{HKST}  for the
definitions of Loewner spaces, and the basic theory of QC mappings between
them. Note that locally, quasiconformality and quasisymmetry are equivalent
in this setting \cite[Theorem 9.8]{HKST}, hence our casual conflation of the
terms when we discuss mappings in these spaces. Thus we have the following
consequence of Theorem \ref{thm:equaldim}.

\begin{corollary}
Let $f:X\to Y$ be a quasiconformal mapping between Ahlfors $D$-regular,
$D$-Loewner spaces, $D>1$. Then for every $0 < d \leq D$ we have
\begin{equation}
  \dim_H f(E) = \dim_H E,
\end{equation}
for $\mod_{D/d}$-almost every bounded Ahlfors $d$-regular set $E\subset X$.
\end{corollary}

We do not know if the equality in
 Theorem \ref{thm:equaldim} holds for all quasisymmetric
mappings between Ahlfors regular spaces. However, we will prove that an
inequality does hold generally.

\begin{theorem}
\label{thm:nonexpand} If  $f:X\to Y$ is a quasisymmetric mapping of Ahlfors
$D$-regular spaces, $D>1$, then for every $0 < d\leq D$ we have
\begin{equation}
\label{fundamentalinequality}
  \dim_H f(E) \leq \dim_H E,
\end{equation}
for $\mod_{D/d}$-almost every bounded Ahlfors $d$-regular set $E\subset X$.
\end{theorem} 

Thus ``generic non-expansion'' holds true for every quasisymmetric mapping
between arbitrary Ahlfors regular spaces. Theorem \ref{thm:nonexpand} is a
special case of Corollary \ref{thm:compression}, which itself is an immediate
Corollary of Theorem \ref{theorem:expansion}; the latter is the main result
of the first half of the paper and states  that ``non-expansion" holds under
much more relaxed conditions than those in Theorem \ref{thm:nonexpand}.

\subsection{Exceptional ``fibers" in product spaces}

For a quasiconformal map of $\mathbb{R}^N$, the ACL condition implies
\begin{equation}\label{eq:measure-unrect}
  \calH^{N-1}(\{y\in\mathbb{R}^{\bot}: \dim f(y+\mathbb{R}) >1 \})=0.
\end{equation}

Dimension distortion of ``generic subspaces" of a Euclidean space under
quasiconformal mappings has recently been studied by Balogh, Monti and Tyson
in \cite{Tyson:frequency}, along with similar explorations for Sobolev
mappings $f\in W^{1,p}(\mathbb{R}^N, Y)$,  $p > N$, where $Y$ is a metric
space. In particular they considered the size of the ``exceptional" family of
parallel $n$ dimensional subspaces whose image experiences a prespecified
jump in dimension. More precisely, it was proved in \cite{Tyson:frequency}
that if $n$ is an integer between $1$ and $N$ and $n'\geq n$ then
  \begin{equation}\label{ineq:Tyson-frequency}
    \calH^{\frac{n}{n'}N-n } \{y\in (\mathbb{R}^n)^{\bot}: \dim f(y+\mathbb{R}^n)> n' \} = 0.
  \end{equation}
Thus the Hausdorff dimension of the ``exceptional" $n$-dimensional subspaces
of $\mathbb{R}^N$ whose dimension may jump over $n'>n$ is at most
$\frac{n}{n'}N-n$, which is strictly less than $N-n$.

The methods in \cite{Tyson:frequency} relied heavily on properties of
Euclidean space, particularly the foliation by affine subspaces, Lipschitz
retractions, the Besicovitch covering theorem, and so forth, that need not
hold in greater generality.  As a result, the results there assume that the
source space be a Euclidean domain.  The authors concluded by asking
\cite[Problem 6.5]{Tyson:frequency} what could be said for more general
spaces, providing the broad motivation for the present paper, which addresses
this question in the setting of quasisymmetric maps between metric spaces.

The questions about dimension distortion of ``generic subspaces" for more
general source spaces, e.g. for the Heisenberg group, have been further
investigated by Balogh, Mattila, Tyson and Wildrick, cf.
\cite{BMT:Grassmanian},\cite{BTW:dimdistortion},\cite{BTW:heisenberg}. In
fact, in \cite{BTW:dimdistortion} and \cite{BTW:heisenberg} a general form of
inequality (\ref{ineq:Tyson-frequency}) was obtained for quasiconformal (and
Sobolev) mappings defined on metric spaces supporting Poincar\'e
inequalities, e.g. the Heisenberg group. Our distortion estimates allow us to
generalize (\ref{ineq:Tyson-frequency}) to even more general product spaces,
(thus no Poincar\'e inequality or even connectivity of the source space is
assumed).





\begin{theorem} \label{expandsmetricquantitative}
Suppose  $E$ is Ahlfors $d$-regular, $F$ is doubling, and  
$f$ is a quasisymmetric map on $E\times F$ with $\dim_H(f(E\times F))\leq
D'$.  Then  for every number $d'\leq D'$,
\begin{equation}\label{badsetnullset}
  \calH^{\frac{d}{d'}D'-d}(\{y\in F:\dim_H f(E\times\{y\})>d'\}))=0\text{.}
\end{equation}
\end{theorem}

Theorem \ref{expandsmetricquantitative} is proven in Section
\ref{Section:products}. Note that for $E=\mathbb{R}^n$ and
$F=\mathbb{R}^{N-n}$, and $f\colon \mathbb{R}^N\rightarrow \mathbb{R}^N$
quasiconformal, we recover the result of Balogh, Monti and Tyson
(\ref{ineq:Tyson-frequency}). However, note also that (\ref{badsetnullset})
not only generalizes (\ref{ineq:Tyson-frequency}) by allowing the set $E$ to
have arbitrary dimension $d$, but also by not requiring to have $d'\geq d$.
The latter inequality of course holds necessarily in the case of
(\ref{ineq:Tyson-frequency}), since $\dim_H(f(y+\mathbb{R}^n )) \geq n$ for
every homeomorphism. In our case, on the other hand, it is possible that
$\dim_H (f(E\times \{y\}))< d$ and moreover $\dim_H (f(E\times F))\leq
D'<\dim_H (E\times F)$ in which case Theorem \ref{expandsmetricquantitative}
estimates the dimension of $y$'s for which $\dim_H f(E\times\{y\})$ is
greater than some $d'<\dim_H f(E\times F)$. For instance, let $E$ be the
$1/2$ dimensional ``$1/4$-Cantor set" obtained from the unit interval $[0,1]$
by dividing it into $4$ equal intervals, removing the middle two intervals
and repeating the process with the remaining intervals. Then $E\times E$ is
the well known four-corner Cantor set of dimension $1$, which has conformal
dimension $0$, see e.g. \cite[Theorem 1.4]{Mackay} or \cite{Tyson:assouad},
and therefore there is a quasisymmetric map  $f$ such that  $\dim_H f(E\times
E)\leq 1$. Theorem \ref{expandsmetricquantitative} then implies that the set
of $y$'s such that $\dim_H f(E\times\{y\})=1$ is $0$-dimensional.

As a consequence of Theorem \ref{expandsmetricquantitative}, we will prove in
Section \ref{Section:products} the following  bound on the infimal dimension
distortion of the fibers.

\begin{corollary} \label{expandsmetric}
Let $E$, $F$, and $f$ satisfy the assumptions of Theorem
\ref{expandsmetricquantitative}. Then
\begin{equation}\label{expands metric}
  \inf_{y\in F}\dim_H f(E\times\{y\})\leq
    \frac{\dim_H f(E\times F)}{\dim_H(E\times F)}\cdot \dim_H (E)\text{.}
\end{equation}
\end{corollary}

Corollary \ref{expandsmetric} is already interesting when $f\colon \mathbb
R^2\rightarrow \mathbb R^2$ is quasiconformal, and $E$ and $F$ lie in the
coordinate axes.

For example, consider a Borel set $S\subseteq \mathbb R$, with $t=\dim_H(S)$.
Applying the corollary to the case $E= {\mathbb R}$ and $F=S$,  inequality
(\ref{expands metric}) becomes
\begin{equation}
\label{planarexample1}
    {\inf_{y \in S} \dim_H f({\mathbb R} \times\{y\})} \leq
    \frac{2}{ t+1}.
\end{equation}

On the other hand, if we additionally assume that $S$ is Ahlfors $t$-regular,
then we may apply Corollary \ref{expandsmetric} with $E=S$ and $F= \mathbb
R$, and obtain (interchanging the order of the factors)
\begin{equation}
\label{planarexample2}
    \inf_{x \in {\mathbb R}}  \dim_H f( \{x\} \times S) \leq
    \frac{2 t }{  t+1 }.
\end{equation}

\subsection{Sharpness of dimension distortion bounds in the plane}

Theorem \ref{expandsmetricquantitative} and Corollary \ref{expandsmetric} are
quite sharp in the planar case.  Our next result establishes the optimality
of  estimates \eqref{planarexample1} and \eqref{planarexample2} in one fell
swoop, as consequences of a stronger result.

\begin{theorem} \label{first thm}
For every $0< t < 1$, there is an Ahlfors $t$-regular Cantor set $S\subseteq
\mathbb R$ such that for each $\epsilon>0$, there is a quasiconformal mapping
$f\colon \mathbb R^2\rightarrow \mathbb R^2$, so that for every Borel subset
$A\subseteq \mathbb R\times S$,
\begin{equation}
\label{megasharp}
\dim_H(f(A))\geq (1-\epsilon)\frac{2\dim_H(A)}{t+1}\text{.}
\end{equation}

In particular,
  \begin{eqnarray}
   \label{first estimate}   \inf_{y\in S}\dim_H f(I \times \{y\}) &\geq& (1 - \epsilon)\frac{2}{t+1} ,\\
   \label{second estimate}   \inf_{x\in I } \dim_H f(\{x\}\times S) &\geq& (1 - \epsilon)\frac{2t}{t+1},
  \end{eqnarray}
for every  interval $I \subset \mathbb{R}$. 

\end{theorem}

This will be proven in Section \ref{Section:sharpnessproofs}. As a corollary
to Theorem \ref{first thm}, we will also obtain the sharpness of
\eqref{ineq:Tyson-frequency} for $N=2$ and $n=1$, and answer \cite[Problem
6.3]{Tyson:frequency} in the affirmative for the planar case.

\begin{corollary}\label{cor:sharp-dimension}
  For every $d'>1$ and $\eta>0$ there is a quasiconformal map $f$ of the plane such that
  \begin{equation}
    \dim_H\{y\in \mathbb{R}: \dim_H f([a,b]\times \{y\})> d', \forall [a,b]\subseteq \mathbb R\} \geq \frac{2}{d'}-1-\eta.
  \end{equation}
\end{corollary}

\subsection*{Sharpness of rectifiability for images of lines}
It is instructive to consider the preceding results in the context of lines
parallel to the $x$-axis.  By the ACL property for quasiconformal maps, if
$f([0,1]\times \{y\}$ has infinite length for every $y\in S\subseteq [0,1]$,
then $\calH^1(S)=0$. It is known that this is rather sharp; Heinonen and
Rhode showed there are examples where $\dim_H(S)=1$ \cite[Theorem 1.9]{HR}.
Even so, in that result the images  $f(\mathbb{R}+y)$ contain many
rectifiable subarcs.  The question of how many lines have purely
unrectifiable images, i.e., images containing no rectifiable subarcs, has
proved to be a much thornier matter. Kovalev and Onninen
\cite{Kovalev:variation} proved that given any countable collection
${\mathcal L}$  of parallel lines in $\mathbb{R}^2$ there is planar
quasiconformal image of ${\mathcal L}$ that contains no rectifiable subarcs,
but until now, it was not known if this could be extended to uncountable
families.   Problem $6.4$ of \cite{Tyson:frequency} and Problem $5.3$ of
\cite{AimPL} ask if this can even be improved at all, i.e., whether there is
even a single uncountable family of lines with this property.

Corollary \ref{cor:sharp-dimension} above already answers this question with
a resounding yes; if, for every $y\in S$, we require that $f(\mathbb
R\times\{y\})$  have no rectifiable subarcs, then the corollary tells us not
only that $S$ can be uncountable, but that $\dim_H(S)$ can  be arbitrarily
close to $1$, and furthermore,  not only may the images of subarcs be
unrectifiable, but they can be taken to have dimension uniformly bounded away
from $1$.  Moreover, from Theorem \ref{first thm} itself, we see that the
dimension distortion factor may be very strongly uniform -- it can be taken
to apply not only to subintervals, but to arbitrary Borel subsets of the
product $S\times \mathbb R$.

If we insist on uniform dimension expansion, then Corollary
\ref{expandsmetric} shows, via estimate \eqref{planarexample1}, that this is
the best we can do --- the set $S$ in the preceding result cannot have
dimension $1$, in contrast to \cite[Theorem 1.9]{HR}.

Our next result shows that if we sacrifice dimension distortion, and merely
ask for unrectifiability of subarcs of the images of lines, then $S$ can
indeed have dimension $1$, and can in a certain sense be as large as
possible, giving a different, yet equally vociferous, ``yes'' to
\cite[Problem 6.4]{Tyson:frequency}.

\begin{theorem}\label{thm:unrectifiable}
For every increasing function $h$ on $[0, \infty)$ such that
\begin{equation}
\label{gaugelimit}
\limsup_{t \to 0}    \frac {h(t)}{t} = \infty\text{,}
\end{equation}
there is a compact set $S \subset [0,1]$ and a quasiconformal map $f$ so that
\begin{enumerate}
\item The quasiconformal constant of $f$ is bounded independent of $h$ and
    $S$.
\item $S$ has infinite Hausdorff measure with respect to $h$ (cf.\
    Definition \ref{def:hmeas} below).
\item $f([0,1]\times \{y\})$ contains no rectifiable subarc for any $y \in
    S$.
\end{enumerate}
\end{theorem}

Theorem \ref{thm:unrectifiable} will be proven in Section
\ref{sec:unrectifiable}. In particular, taking $h(t)=t|\log t|$ in the
preceding theorem produces a compact set $S \subset [0,1]$ of Hausdorff
dimension
 $1$  and a quasiconformal
map $f$ of the plane so that $f([0,1]\times \{y\})$ contains
 no rectifiable subarcs for any $y \in S$.

Note that by the ACL property, the condition \eqref{gaugelimit} on the gauge function $h$ is sharp. 


\subsection*{Conformal dimension}
Rewriting inequality (\ref{expands metric}) as follows
\begin{equation}\label{ratios1}
  \inf_{y\in F} \frac{\dim_H f(E\times \{y \})} { \dim_H (E \times \{y\})}
 \leq \frac{\dim_H f(E\times F)}{\dim_H (E\times F)},
\end{equation}
we obtain the principle {\emph{``fiberwise expansion implies global
expansion"}}: if every  fiber $E \times \{y\} $ has its dimension increased
by a factor $\alpha \geq 1$, then the dimension of the whole product $E\times
F$ increases by at least a factor of $\alpha$ as well.

This principle has an immediate implication when considering conformal
dimension.  Recall that the \textit{conformal dimension} of a metric space is
the infimal Hausdorff dimension of its image under any quasisymmetric map,
i.e.,
$$\dim_C X = \inf_{f\in QS(X)} \dim_H f(X) \text{,}$$
where $QS(X)$ is the class of all quasisymmetric maps on $X$. In the event
that $E$ is minimal for conformal dimension (i.e., $\dim_H E = \dim_C E$),
the inequality \eqref{ratios1} then implies $E\times F$ is minimal as well.
When  $E=\mathbb{R}$ this gives a well known result of Tyson
\cite{Tyson:minimality}. Also see the discussion after Corollary
\ref{thm:compression} for more general results in this vein.

\begin{remark}
We cannot reverse (\ref{ratios1}) by replacing the infimum by a supremum.
   Consider a quasiconformal map $f$ that maps $\{0\}\times \reals$ to a curve
   of dimension $D >1$, but is smooth elsewhere. If $F = \reals$ and
   $E\subset \reals$ contains $0$ and has dimension $0 < d < D-1$,
   then the right side of (\ref{ratios1}) is at least $D/(d+1) > 1$, but the
   smoothness of $f$ off of $\{0\}\times \reals$ implies
   $\dim f(E \times \{y\}) / \dim(E \times \{y\})=1 $ for every $y$.
In other words {\textit{``global expansion does not imply fiberwise
expansion"}}.
\end{remark}

\vskip .2cm This paper is organized as follows. In Section
\ref{Section:preliminaries} we review the necessary definitions and
preliminary results needed in the paper. In Section \ref{section:moduli} we
define the various versions of modulus and prove Theorem \ref{thm:carnot}
assuming Theorem \ref{thm:equaldim}. Theorem \ref{thm:randomtranslates} is a
particular case of Theorem \ref{thm:carnot}. In Section \ref{Section:results}
we state more general versions of Theorem \ref{thm:nonexpand}, and some
corollaries; the reader will observe that Theorem \ref{thm:nonexpand} is a
special case of Corollary \ref{thm:compression}. In Sections
\ref{section:nonexpansionproofs}, \ref{Section:products} and
\ref{Section:sharpnessproofs} we prove the remaining theorems and corollaries
given in this introduction. In Section \ref{Section:remarks&corollaries} we
list some open problems.

\subsection*{Acknowledgements}
Work on this paper was started after the second author visited the workshop
``Mapping theory in metric spaces" organized by Luca Capogna, Jeremy Tyson,
and Stefan Wenger, which was held at the American Institute of Mathematics.
We thank the institute for its hospitality. We would also like to thank
Pietro Poggi-Corradini and Jeremy Tyson for numerous comments and discussions
about the paper. We are particularly grateful to Leonid Kovalev for careful
reading and very detailed comments which greatly improved the paper. We also
thank the referee for a careful reading of the  manuscript and a thoughtful
report that included numerous suggestions that improved the paper.

\section{Measures and mappings}\label{Section:preliminaries}

\subsection{Measures, dimension and Ahlfors regularity}
Unless otherwise stated, the metric spaces in this paper are assumed to be
separable. Given a metric space $X=(X,d_X)$ we will denote by $B(x,r)$ the
closed ball in $X$ of radius $r$ centered at $x$. When $B=B(x,r)$ and $k>0$,
we denote by $kB$ the ball $B(x,kr)$. If the metric space $X$ is clear from
the context we will often denote the metric $d_X$ by $d$.

\begin{lemma}[\textbf{Covering Lemma}]\label{covlemma}
Every family $\calB$ of balls of bounded diameter in a compact metric space
$X$ contains a countable subfamily of disjoint balls $B_i\subset\calB$ such
that
$$\bigcup_{B\in\calB}B\subset\bigcup_{i}5B_i.$$
\end{lemma}

The space $X$ is said to be \textit{doubling} if there is a constant $C$ such
that every ball in $X$ may be covered by  $C$ balls of half the radius.

Throughout the paper, the term ``measure'' refers to a Borel regular outer
measure.  A measure $\lambda$ on $X$ is \textit{locally finite} if  every
$x\in X$ lies in a neighborhood $U\subset X$ with $\lambda(U)<\infty$.
$\lambda$ is \textit{doubling} if there is a constant $C$ such that for each
ball $B\subseteq X$, $\lambda(2B)\leq C\lambda(B)$.

We are particularly interested in (generalized) Hausdorff measures, whose
definition we now recall.
\begin{definition}
\label{def:hmeas} Given a non-negative function $h:[0,\infty)\to[0,\infty)$,
and a subset $E\subseteq X$, the Hausdorff $h$-measure of $E$ is defined as
follows. For every $\eps\in (0,\infty]$, let
\begin{equation*}
\calH_{\eps}^{h}(E)=\inf{ \left\{
\sum_{i=1}^{\infty}h(r_i):
E\subset\bigcup_{i=1}^{\infty}B(x_i, r_i), \, r_i<\varepsilon
\right\} },
\end{equation*}
and $$\calH^{h}(E)=\lim_{\eps\to0} \calH^h_{\eps}(E).$$
\end{definition}
When $h(r)=r^t, t\geq0$, the resulting measure is called the
\textit{$t$-dimensional Hausdorff measure} and is denoted by $\calH^t$.
 The
\textit{Hausdorff dimension} of $X$ is
\begin{eqnarray*}
\dim_H(X)=\inf\{\,t\,: \calH^t(X)=0\}.
\end{eqnarray*}

A subset $E\subset X$ is \textit{\textsf{Ahlfors $d$-regular}}, if there is a
constant $C_E\geq 1$ such that for every $x\in E$ and $0<r<\diam E$ the
following inequalities hold
\begin{equation}\label{ineq:Ahlfors}
  \frac{1}{C_E} r^d \leq \calH^d(E\cap B(x,r))\leq C_E r^d.
\end{equation}
We will denote by $\calA_d(X)$ the collection of all bounded Ahlfors
$d$-regular subsets of $X$, that is for every $E\in\calA_d(X)$  the
inequalities (\ref{ineq:Ahlfors}) hold for some constant $C_E\geq1$.

We will denote the support of a measure $\lambda$ on $X$ by $E_\lambda$.  We
say that $\lambda$ is  \textit{\textsf{Ahlfors $d$-regular}}, $d>0$, if there are
constants $C_\lambda\geq 1$ and $r_\lambda>0$ such that
\begin{equation}\label{ineq:p-regular}
C^{-1}_{\lambda} r^{d} \leq \lambda(B(x,r))\leq C_{\lambda} r^{d},
\end{equation}
for every ball $B(x,r)$ centered at $x\in E_\lambda$ with radius
$r<r_\lambda$. More generally, we say $\lambda$ is \textsf{\emph{upper or
lower $d$-regular}} if only the right or left inequality in
(\ref{ineq:p-regular}) holds, respectively.  We denote by  $\calL_d(X)$ the
family of lower $d$-regular measures in $X$.

\begin{remark}
\label{rem:regularityismetric} If $\lambda$ is $d$-regular, then so is the
restricted Hausdorff measure $\calH^d\lfloor_{E_\lambda}$  \cite[Exercise
8.11]{H}, so that if we were only interested in the two-sided regularity
condition, there would be no special reason to consider  $d$-regular measures
rather than sets.  On the other hand, by itself, mere upper (resp.\ lower)
regularity of $\lambda$ does not imply the same condition for
$\calH^d\lfloor_{E_\lambda}$. Whereas the full two-sided Ahlfors regularity
condition is rather strong, the existence of upper and lower regular measures
holds in rather great generality ---  the former may be obtained via the
Frostman Lemma (Lemma \ref{lem:frostman} below), and the latter exist on
compact doubling metric spaces, via a theorem of Vol'berg and Konyagin
\cite{VolbergKonyagin}.\footnote{Note that the results in
\cite{VolbergKonyagin} are formulated in terms of ``homogeneous'' measures,
but on a bounded set such measures are easily seen to be lower regular.

} 

The upper regular measures given by the Frostman Lemma are crucial to our
applications to product spaces in Theorem \ref{expandsmetricquantitative} and
Corollary \ref{expandsmetric}. Also, the arc-length measure of a curve
satisfies the lower regularity condition, though not necessarily the upper,
so that in order to view our results as a generalization of facts about curve
modulus, we must consider lower regular measures. It is for these reasons
that our most general results, Theorem \ref{theorem:expansion} and Corollary
\ref{cor:expansion}, are formulated in terms of families of measures, not
sets.
\end{remark} 



We refer to \cite{Mattila} and \cite{H} for proofs of the next two lemmas and
for further discussion of Hausdorff measures, dimension and Ahlfors regular
spaces and their properties.

\begin{lemma}[Mass distribution principle]\label{masslemma}
If the metric space $X$ supports a positive upper $D$-regular Borel measure,
then $\calH^D(X)>0$.  In particular, $\dim_H(X)\geq{D}$.
\end{lemma}

An important converse to the mass distribution principle is  the following
lemma, see \cite[Theorem $8.8$]{Mattila}.

\begin{lemma}[Frostman's Lemma]
\label{lem:frostman}
If $X$ is a doubling metric space, and $F\subseteq X$ is a Borel set such that $\calH^s(F)>0$, then there is an upper $s$-regular  measure $\nu$ supported on $F$ such that $\nu(F)>0$. 
\end{lemma}

\begin{remark}
Frostman's Lemma is often stated for the special case $X=\mathbb R^n$.
However, even if $(X,d_X)$ is only a doubling metric space, the lemma is easily obtained from the Euclidean case via Assouad's embedding theorem \cite{Assouad}. For simplicity we will denote by $X^t$ the metric space $(X,d_X^t)$ for $t\in(0,1)$, i.e. the $t$-snowflaked version of $X$. Now, if $\calH^s(F)>0$ for some $F\subset X$ then
by Assouad's embedding theorem \cite{Assouad} for every $s'>s$ there is a bi-Lipschitz embedding $j:X^{s/s'}\hookrightarrow\mathbb{R}^n$ for some $n\in\mathbb{N}$. Since $j$ is bi-Lipschitz, we have
$$\calH^{s'}(j(F^{s/s'}))\cong\calH^{s'}(F^{s/s'}) = \calH^s(F)>0,$$
where $F=(F,d_X|_F)$. Assuming Frostman's Lemma for $\mathbb{R}^n$, we have that  $j(F^{s/s'})\subset\mathbb{R}^n$  supports a nontrivial, upper $s'$-regular measure. Since $j$ is bi-Lipschitz, the space $F^{s/s'}$ also supports an upper $s'$-regular measure $\mu$ such that $\mu(F^{s/s'})>0$. Then, if $\nu$ is defined on $F$ as the pullback of $\mu$ under the snowflaking, i.e. $\nu(E)=\mu(E^{s/s'})$ for every Borel set $E\subset F$, then for every ball $B_X(x,r)\subset X$ we have
\begin{align*}
  B_X(x,r) & = \{y \in X : d_X(x,y)<r \}  = \{y \in X : (d_X(x,y))^{s/s'}<r^{s/s'} \} \\
 & = \{y \in X^{s/s'} : d_{X^{s/s'}}(x,y) < r^{s/s'} \} = B_{X^{s/s'}}(x,r^{s/s'}),
\end{align*}
and therefore
\begin{align*}
  \nu(B_X(x,r))
  = \mu(B_{X^{s/s'}}(x,r^{s/s'}))
  \leq C(r^{s/s'})^{s'} = Cr^s.
\end{align*}
Thus $F$ supports an upper $s$-regular measure $\nu$, such that $\nu(F)=\mu(F^{s/s'})>0$.
\end{remark}




\subsection{Quasiconformal and quasisymmetric mappings}
Given a homeomorphism $f:X\to Y$, $x\in X$ and $r>0$ we let
\begin{eqnarray*}
  H_f(x,r)=\displaystyle{\frac{{\sup_{d_X(x,y)\leq r}} d_Y(f(x),f(y))}{\inf_{d_X(x,y)\geq r} d_Y(f(x),f(y))}}.
\end{eqnarray*}
The mapping $f$ is called (metrically) \textit{quasiconformal} if there is a
constant $H<\infty$ such that
\begin{equation}\label{definition:metricQC}
  \limsup_{r\to0}H_f(x,r) \leq H
\end{equation}
for every $x\in X$.

Because quasiconformality is an infinitesimal property it is often hard to
work with directly. For this reason one often requires a stronger, global
condition from a mapping $f$, which we discuss next.

Let $\eta:[0,\infty)\rightarrow[0,\infty)$ be a fixed homeomorphism.  A
homeomorphism $f$  between metric spaces $(X,d_X)$ and $(Y,d_Y)$ is called
$\eta$-\textit{quasisymmetric} if for all distinct triples $x,y,z\in{X}$ we
have
\begin{equation}\label{QS}
\frac{d_Y(f(x),f(y))}{d_Y(f(y),f(z))}\leq{\eta\left(\frac{d_X(x,y)}{d_X(y,z)}\right)}.
\end{equation}

Quasisymmetric mappings do not distort (macroscopic) shapes too much. In
particular an image of a round ball will be ``roundish", a condition which a
priori holds for QC maps only on small scales depending on $x$.

A mapping $f$ of a metric space $(X,d_X)$ into $(Y,d_Y)$ is called a
\emph{quasisymmetric embedding} if $f$ is a quasisymmetric map of $X$ onto
$f(X)\subseteq Y$, where the metric on $f(X)$ is the restriction of the
metric of $Y$.


It follows almost immediately from the definition that quasisymmetric maps do not distort annuli too much.  More precisely, we have the following easy result, which is very similar to Lemma 3.1 in \cite{Tyson:qcqs}, and
which we will use in the proof of Theorem \ref{theorem:expansion} below.
\begin{lemma}\label{lemma:qsdistortion}
If $f$ is an $\eta$-quasisymmetric map of a separable metric space $X$, then
for every closed ball $B'=B(y,s)\subset f(X)$, with $y=f(x)$, there is a
closed ball $B=B(x,r)$ such that for each $k\geq 1$, the following inclusions
hold:
\begin{equation*}
B'\subseteq f(B)\subseteq f(kB)\subseteq \eta(k)B'\text{.}
\end{equation*}
Furthermore, if $f^{-1}$ is uniformly continuous, with modulus of continuity
$\omega(t)$, then we may choose $r$ so that $r\leq\omega(s)$.
\end{lemma}

\begin{proof}
Let $r>0$ be the smallest number such that the first inclusion holds.  Then
for each $\alpha<1$, and a point $x_\alpha\in f^{-1}(B')\backslash \alpha B$,
we let $y_\alpha=f(x_\alpha)\in B'\backslash f(\alpha B)$. Therefore,
whenever $y_1=f(x_1)$, with $x_1\in kB$, we have
\[
d(y,y_1)\leq  d(y,y_\alpha) \eta\left(\frac{d(x,x_1)}{d(x,x_\alpha)}\right) \leq s \eta\left(\frac{k}{\alpha }\right)\text{.}
\]
Passing to the limit as $\alpha$ goes to $1$, we see that $d(y,y_1)\leq
s\eta(k)$, from which the last inclusion follows.  Finally,  when $f^{-1}$
has modulus of continuity $\omega$, the choice of $r$ implies  $r\leq
\omega(s)$.
\end{proof}


As a consequence of  Lemma \ref{lemma:qsdistortion} we obtain the following
version of the mass distribution principle where we require an upper estimate
of $\mu$ only for a limited collection of subsets.

\begin{lemma}\label{upgradedmasslemma}
 Let $f$ be a quasisymmetric map of a metric space $X$.  If there exist
constants $C_1,C_2\geq 1$, an integer $L\in\mathbb{N}$ and a measure $\mu$ on
$f(X)$ such that
\begin{itemize}
  \item there is a collection of sets $\mathcal{Q} =
      \{Q_i\}_{i\in\mathbb{N}}$ s.t. every ball $B=B(x,r)$ can be covered
      by $L$ members of $\mathcal{Q}$, $Q_{i_1},\ldots,Q_{i_L}$, so that
      for $k=1,\ldots,L$ we have
      $${\diam Q_{i_k}} \leq C_1{r},$$
  \item for every $i\geq 1$ we have $\mu(f(Q_i))\leq C_2 \diam(f(Q_i))^s$,
      for some $s>0$,
\end{itemize}
then $\dim_H f(X)\geq s$.
\end{lemma}

\begin{proof}
 Let $B'=B(f(x),R)$. We want to show that $\mu(B')\lesssim R^s$. Let $B$ be
the ball containing $f^{-1}(B')$ given by Lemma \ref{lemma:qsdistortion}.
Since $Q_{i_k}\subset (1+C_1)B$, from Lemma \ref{lemma:qsdistortion} we have
for $k=1,\ldots, L$
$$f(Q_{i_k})\subset f((1+C_1) B) \subset \eta(1+C_1) B'.$$
Therefore
 \begin{align*}
   \mu(B')
   &\leq \mu(f(B)) \leq \mu (\bigcup_{k=1}^L f(Q_{i_k}))\leq \sum_{k=1}^L \mu(f(Q_{i_k}))\\
   &\leq C_2 \sum_{k=1}^L (\diam f(Q_{i_k}))^s
   \leq C_2 \sum_{k=1}^L (\eta(1+C_1)\cdot2R)^s
   = C R^s,
 \end{align*}
where $C=C_2 L [2\eta(1+C_1)]^s$. Applying the mass distribution principle
completes the proof.
\end{proof}

%

Even though quasisymmetry is a stronger condition than quasiconformality, it
is often the case that the two notions coincide. For instance if a
homeomorphism $f:\mathbb{R}^N\to\mathbb{R}^N, N\geq 2$ is QC then it is also
QS. This was proved by Gehring for $N=2$ in \cite{Gehring:definitions} and by
V\"ais\"al\"a for $N\geq3$ \cite{Vaisala:definitions}. More recently Heinonen
and Koskela extended this equivalence to a large class of metric spaces
\cite{HK:Acta},\cite{HK}.

\section{Modulus}\label{section:moduli}
The main tool used in this paper is the modulus of a family of curves, sets,
or measures. In this section we review the definitions and basic properties
for each of these concepts. In the case when the underlying measure space is
a locally compact topological group $G$, and $\la$ is any measure on $G$, we
estimate from below the modulus of the family of translates of $\la$ by
elements $g\in K$, where $K$ is any subset of $G$. This estimate is vital for
the proof of Theorem \ref{thm:carnot}.

\subsection{Modulus of curve families} Given a metric measure space $(X,\mu)$, a family of curves $\G$ in $X$
and a real number $p\geq1$ the $p$-\emph{modulus} of $\G$ is defined as
$$\m_p \G = \inf_{\rho} \int_X \rho^p d\mu,$$
where the infimum is taken over all $\G$-admissible nonnegative Borel
functions $\rho$.  Here a function $\rho:X\to[0,\infty)$ is
\emph{$\G$-admissible} if $\int_{\g}\rho d s\geq 1$ for every locally
rectifiable curve $\gamma\in\G$, where $ds$ denotes the arclength element. We
say that a property holds for \emph{$p$-almost every curve} in $X$ if it fails
only for a curve family $\G$ such that $\m_{p}\G=0$. Notice that by
definition, almost every curve is locally rectifiable. We refer to ~\cite{H}
and ~\cite{HK:Acta} for further details on modulus of curve families
including the definitions of rectifiability and arclength in general metric
spaces.

Despite superficial differences,  proofs of the aforementioned equivalence
between the definitions of metric quasiconformality and  quasisymmetry, no
matter the level of generality, tend to follow the same broad outline:  the
metric definition is used, with the help of various covering arguments, to
establish quasi-invariance of the conformal modulus of path families, and
this invariance, along with modulus estimates for certain families,
facilitates geometric arguments that yield the global distortion estimate
\eqref{QS}.

As a result, one obtains the following equivalent definition of
quasiconformality, the so-called  ``geometric" definition.

Given $D>1$ and $K\geq 1$, a homeomorphism $f\colon X\to Y$ between Ahlfors
$D$-regular spaces is called \textit{(geometrically) $K$-quasiconformal} if
for every family of curves $\G$ in $X$ the following inequalities hold
\begin{eqnarray}
  {K^{-1}}\m_D f(\G) \leq \m_D \G \leq K \m_D f(\G),
\end{eqnarray}
where $f(\G)$ denotes the image of the family $\G$ under $f$,  see
\cite{Ahlfors:QCmaps},\cite{Vaisala:lectures}.

For proofs of the equivalence of geometric quasiconformality to the metric
definition, and to quasisymmetry, in various levels of generality, we refer
the reader again to \cite{Gehring:definitions, HK:Acta, HK,
Vaisala:definitions}.



\subsection{Modulus of families of measures}
The notion of modulus can be extended far beyond the context of curve
families.  The modulus of a family of measures, with respect to an underlying
measure, was defined and studied by Fuglede in \cite{Fug} and by Ziemer in
\cite{Ziemer}. In \cite{Hak:IMRN} the second author used Fuglede's modulus to
study conformal dimension of various spaces. More recently, Badger studied
extremal metrics and Beurling criterion for families of measures
\cite{Badger:Fuglede}.

Let $(X,\mu)$  be a metric measure space and $p\geq 1$. Let $\calL$ be a
collection of measures on $X$.  A Borel  function $\rho:X\to [0,\infty]$ is
said to be admissible for $\calL$ if
$$\int_X \rho d\la\geq 1$$
for every $\la\in\calL$. The \emph{$p$-modulus of $\calL$} is
$$\m_p(\calL,\mu)=\inf_{\rho- \mbox{\tiny{adm}}}\int_X\rho^p d\mu,$$
where the infimum is taken over all $\calL$-admissible functions $\rho$.
Often we simply write $\m_p(\calL)$, when $\mu$ is clear from context.

Next, we summarize some of the properties of modulus that will be useful for
us.

\begin{lemma}\label{lemma:modulus-properties}
For every $p\geq 1$ the following properties hold.
  \begin{enumerate}
    \item $\mathrm{\textsc{(Monotonicity)}}$ $\m_p\calL \leq \m_p\calL'$,
        if $\calL\subset\calL'$
    \item $\mathrm{\textsc{(Subadditivity)}}$ $\m_p\calL \leq
        \sum_i\m_p\calL_i$, if $\calL=\bigcup_{i=1}^{\infty}\calL_i$
    \item $\mathrm{\textsc{(Ziemer's Lemma)}}$  If $1<p<\infty$, $\calL_1\subset\calL_2\subset\ldots$ are families of 
    measures and $\calL=\cup_{i=1}^{\infty}\calL_i$ then $\m_p \calL =
    \lim_{i\to\infty} \m_p \calL_i.$
  \end{enumerate}
\end{lemma}

See \cite{Fug} for (1) and (2). Property (3) is due to Ziemer for families of
continua  in $\mathbb{R}^N$, see Lemma $2.3$ in \cite{Ziemer}.  The  proof in
the case of general measure families is the same as in \cite{Ziemer}. It is
important to emphasize here that Ziemer's Lemma holds only under  the
assumption $p>1$.

We say a property holds for \textsf{\emph{$p$-almost every
$\lambda\in\calL$}} if it fails only for a family $\calL_0\subset\calL$ such
that $\m_{p}(\calL_0)=0$.

\subsection{Modulus of families of Ahlfors regular sets}\label{section:conformal}
The notion of modulus defined above is quite general, but also a bit
technical and abstract, and presumes we have at hand not merely a single
underlying measure, but a family of other measures as well. In the greatest
generality, this complication is unavoidable --- we cannot speak of the
modulus of a family of sets without some measures on hand to formulate the
admissibility condition. Our motivation for working in such generality was
discussed earlier in Remark \ref{rem:regularityismetric}.

Ahlfors regularity, on the other hand, is fundamentally a metric notion, in
the sense that the existence of any Ahlfors $d$-regular measure on a set $E$
is equivalent to $d$-regularity of the Hausdorff measure
$\calH^d\lfloor_{E}$, and the latter property is determined entirely by the
metric. With this in mind, given any family  $\calE\subseteq \mathcal A_d(X)$
of Ahlfors $d$-regular sets, we define the $p$-modulus of $\calE$ to be
$$\m_{p}(\calE,\mu) = \m_{p} ( \{ \calH^{d}\lfloor_{E} \}_{E\in\calE},\mu)\text{.}$$
As before, when $\mu$ is clear from context, we omit it.  In particular, we
do this when $X$ is $D$-regular and $\mu=\calH^D$.  The dimension $d$ will
always be clear from context, so we do not include it in the notation either,
and in any case, a given set can be Ahlfors $d$-regular for at most one value
of $d$.

Finally, the most important modulus in (quasi)-conformal geometry, when
considering $d$-dimensional subsets of $D$-dimensional spaces, is the
conformal modulus $\m_{D/d}$.  Thus when $\calE\subseteq \mathcal A_d(X)$,
and $X$ is Ahlfors $D$-regular, we unambiguously define $\m(\calE):=
\m_{{D}/{d}}(\calE)$, and simply say ``almost every $E\in \calE$'' to refer
to a property that holds for all $E\in\calE\backslash \calE_0$, for some
subfamily $\calE_0$ with $\m(\calE_0)=0$.

\subsection{Modulus and products}\label{subsection:products}
Let $(E,\la)$ and $(F,\nu)$ be two metric measure spaces with
$0<\la(E)<\infty$. Denote $X=E\times F$ and $\mu=\la\times \nu$. Let
$\mathcal{F}=\{E\times\{y\}: y\in F\}$, and for each $y\in F$, let $\la_y$ be
the pushforward of $\la$ by the map $x\mapsto (x,y)$, so that for each Borel
set $A\subseteq X$,  $\la_y(A)=\la(\{x\in E:(x,y)\in A\})$.

 \begin{lemma}\label{lemma:product modulus}
 With the notation as above,  let
$\calL_F=\{ \la_{y}: y\in F\}$.  Then for every $p\geq 1$ we have
   \begin{equation}
     \m_{p} \calL_F = \frac{\mu(X)}{\la(E)^p} = \frac{\nu(F)}{\la(E)^{p-1}}.
   \end{equation}
 \end{lemma}

\begin{proof}
This proof is the same as in the classical case of curve families.  We give
it here for completeness.
First note that since the function $\rho(x,y)\equiv\la(E)^{-1}$ is admissible
for $\calL_F$, we have $\m_{p} \calL_F \leq \frac{\mu(X)}{\la(E)^p}.$
To obtain the lower bound, note that for every $\calL_F$-admissible $\rho$ we
have $\int_E \rho(x,y) d\la\geq1, \forall y\in{Y},$ and therefore by
H\"older's inequality we obtain that for every $y\in Y$ the following holds
$$1\leq \la(E)^{p-1} \int_{E}\rho^{p}(x,y)d\la.$$
Integrating both sides of this inequality with respect to $\nu$ we obtain
$$\nu(F)\leq \la(E)^{p-1} \int_{X}\rho^{p}(x,y) d\mu,$$
and therefore
$$\frac{\mu(X)}{\la(E)^p} = \frac{\nu(F)\la(E)}{\la(E)^{p}}\leq \int_{X}\rho^{p}(x,y) d\mu.$$
Hence $\m_{p}\calL_F\geq\frac{\mu(X)}{\la(E)^p}$.
\end{proof}

\subsection{Modulus and group translations} In this subsection we consider
another example of a family of measures - a family of translates of a given
measure $\lambda$ by elements of a set $K\subset G$, where $G$ is a
topological group, which in particular could be $\mathbb{R}^n$. We will first
show that the modulus of the family of translates of $\la$ depends on the
measure of $K$, and then will consider an example of translates of the Cantor
set in the real line $\mathbb{R}$.

 Suppose $G=(G,\nu)$ is a locally
compact topological group, with right invariant Haar measure $\nu$.  Let
$\lambda$ be another measure on $G$. For each $y\in G$, denote by
$y_*\lambda$ the pushforward of $\lambda$ by \textit{left} multiplication by
$y$.

If $K\subseteq G$ is measurable, let $K_*\lambda=\{y_*\lambda:y\in K\}$.

\begin{lemma}\label{lem:translatesmodulus}
For every $p\geq 1$ we have
\begin{equation}\label{ineq:main}
\mod_{p}(K_*\lambda,\nu)
\geq
\frac{\nu(K)}{\la(G)^{p}}.
\end{equation}
\end{lemma}

\begin{proof}[Proof of Lemma \ref{lem:translatesmodulus}]
Let $\rho\colon G\to [0,\infty]$ be admissible for the family $K_*\lambda$,
and fix $y\in K$. Since translation by $y$ maps $G$ to itself, we have
$y_*\lambda(G)=\lambda(G)$. Thus by H\"older's inequality,
\[
\int_G\rho(yx)^p\, d\lambda(x)=
\int_G\rho^p\, dy_*\lambda
\geq \frac{\left(\int_G\rho\, dy_*\lambda\right)^p}{\lambda(G)^{p-1}}
\geq \frac{1}{\lambda(G)^{p-1}}\text{.}
\]
Thus we obtain
\begin{align*}
 \lambda(G)\int_G\rho^p\,d\nu
&=\int_G\int_G\rho(y)^p\,d\nu(y)\,d\lambda(x)
=\int_G\int_G\rho(yx)^p\,d\nu(y)\, d\lambda(x)& \text{(right invariance of $\nu$)}\\
&=\int_G\int_G\rho(yx)^p\, d\lambda(x)\,d\nu(y)&\text{(Fubini's Theorem)}\\
&\geq \int_K\int_G\rho(yx)^p\, d\lambda(x)\,d\nu(y) \\
&\geq \int_K \left(\frac{1}{\lambda(G)^{p-1}}\right)\,d\nu=\frac{\nu(K)}{\lambda(G)^{p-1}}\text{.}
\end{align*}
Dividing each side by $\lambda(G)$ and infimizing over all admissible
functions $\rho$, we obtain the desired inequality \eqref{ineq:main}.
\end{proof}

Using the terminology of Subsection \ref{section:conformal} we have the
following consequence of Lemma  \ref{lem:translatesmodulus}.

\begin{corollary}\label{cor:translates}
Let $n\geq 1$, $0<d<n$. If $E$ is a nonempty, bounded Ahlfors $d$-regular
subset of $\mathbb{R}^n$ and $K\subseteq \mathbb{R}^n$ is a Lebesgue
measurable set ``of translates" of $E$, then the family of translates $\{y+E
: y\in K\}$ has positive $p$-modulus, for any $p\geq 1$, whenever $K$ has
positive $n$-dimensional Lebesgue measure. More precisely, if $\calH^n(K)>0$
then
\begin{align*}
 \m_p (\{\calH^d \lfloor_{\, y+E} \, : \, y\in K \}, \calH^n) >0,
\end{align*}
for every $p\geq 1.$
\end{corollary}

\begin{proof}
Let $(G,\nu)=(\mathbb{R}^n,\calH^n)$ and $\la=\calH^d \lfloor_{E}$. Then
$0<\la(G)=\calH^d(E)<\infty$ since $E$ is a nonempty, bounded Ahlfors regular
set. Moreover $y_*\la = \calH^d \lfloor_{\,y+E}$. The result follows
immediately from inequality (\ref{ineq:main}).
\end{proof}

\begin{remark}
Note that the converse of Corollary \ref{cor:translates} is not true; it is
possible to have a set $K\subset\mathbb{R}^2$ of zero Lebesgue measure such
that the family $\{y+E : y\in K\}$ has positive modulus. Indeed, if
$E=[0,1]$, $\la=\calH^1\lfloor_{\,[0,1]}$ and $K$ is the vertical segment of length one, connecting the origin to the point $(0,1)\in\mathbb{R}^2$, then the family of
translates $K_* {\la}$ coincides with the product family $\calL_{[0,1]}$ as
in Lemma \ref{lemma:product modulus}, i.e. with the family of restrictions of
the one-dimensional Lebesgue measure to the horizontal segments
$[0,1]\times\{y\},$  with $y\in[0,1]$. But $\m_p(\calL_{[0,1]},\calH^2) =1 $
for every $p\geq 1$, by Lemma \ref{lemma:product modulus}, even though
$\calH^2(K) = 0.$
\end{remark}

\subsection{Modulus and Minkowski sum} From the previous remark it follows that the positivity of modulus of a
family of translates $K_*\la$ of a measure $\la$ is not characterized by the
measure $\nu(K)$ of the set of translates. However, the example in that
remark may lead the reader to think that one may be able to characterize the
positivity of modulus (at least in $\mathbb{R}^n$) in terms of the measure of
the Minkowski sum $K+E$, i.e. the union of the supports of $y_*\la$ as $y$
runs through $K$. We will show that this is also not true. For this we will
consider the middle thirds Cantor set $C\subset [0,1]$ and the Bernoulli
probability measure $\la$ supported on $C$, and will show that there are two
sets of translates $K_1,K_2\subseteq[0,1]$ and $p\geq 1$ such that
$$\m_p({K_1}_*\la,\calH^1)>0 \quad \mbox{ and } \quad
\m_p({K_2}_*\la,\calH^1) = 0$$ even though
$$K_1+C = K_2+C.$$

Let $K_1 = [0,1]$ then clearly $K_1+C = [0,2]$.
Moreover, by Lemma \ref{lem:translatesmodulus} for every $p\geq 1$ we have
\begin{align*}
  \m_p({K_1}_*\la,\calH^1) \geq \frac{\calH^1([0,1])}{\la(C)^{p-1}}=1>0.
\end{align*}

Let $K_2 = C$. It is well known that $C+C = [0,2]$, see e.g. \cite{Randolph}.
Next we show that $\m_p({K_2}_*\la,\calH^1) = 0$ for some values of $p$.

\begin{lemma}
Let $C\subset I:=[0,1]$ be the middle-thirds Cantor set and $\la$ be the
Bernoulli probability measure on $C$. Then for $1\leq p < \log_2 3$ we have
\begin{align}
  \m_p({C}_*\la,\calH^1) &= 0.
\end{align}
\end{lemma}
\begin{proof}
To find $\m_p({C}_*\la,\calH^1)$ we let
\begin{align*}
\rho_i = 2^i\chi_{[1-\frac{1}{3^{i}},1+\frac{1}{3^{i}}]}
\end{align*}
for every $i\geq 1$. We will show $\rho_i$ is admissible for $C_*\la$ for
every $i\geq 1$. To see that, fix a point $y\in C$ and consider the measure
$y_* \la$ supported on $y+C$. Note, that $1\in y+C$, since $1-y\in C$ if
$y\in C$. Next for $i\geq 1$ let $J_i$ denote the $i$th generation interval of length $3^{-i}$ used in the standard construction of the Cantor set, which contains the point $1-y$. Then, by definition
of $\la$ we have that $\la(J_i)=2^{-i}$. Moreover, since  $y+J_i$ contains the point
$1\in\mathbb{R}$ we also have
$$y+J_i\subset\left[1-\frac{1}{3^i},1+\frac{1}{3^i}\right].$$
Thus,
\begin{align*}
  \int \rho_i d(y_*\la)
  &= 2^i (y_* \la)([1-\frac{1}{3^i},1+\frac{1}{3^i}])
  \geq 2^i (y_* \la) (y+J_i)\\
  &\geq 2^i \la(J_i) = 2^i 2^{-i} =1,
\end{align*}
where we used the fact that $\la(J_i)=2^{-i}$ on every $i$'th generation
interval $J_i$ of the Cantor set $C$. Thus $\rho_i$ is admissible for  $C_*
\la$ for every $i\geq 1$, and we estimate the modulus of $C_*\la$ as follows
\begin{align*}
\m_p({C}_*\la,\calH^1)
& \leq \int_{C+C} \rho_i^{p} d\calH^1  \leq \int_0^2 \left(2^i\chi_{[1-\frac{1}{3^{i}},1+\frac{1}{3^{i}}]}\right)^p d\calH^1
 = 2 \left( \frac{2^p}{3}\right)^i \xrightarrow{i\to\infty} 0,
\end{align*}
if $1\leq p<\frac{\log 3}{\log 2}$. Thus, $\m_p({C}_*\la,\calH^1) = 0$ for
$p\in[1,\log_2 3)$.
\end{proof}

\begin{remark}
Lemma \ref{lem:translatesmodulus} can be generalized to the case where
$\lambda$ is a measure on the semigroup $\mathcal F_\nu$ of $\nu$-preserving
transformations on a measure space $(Y,\nu)$.  In the above proof, instead of
taking $x,y\in G$, one takes $\phi\in \mathcal F_\nu$, $y\in Y$, and
integrates accordingly, replacing $yx$ with $\phi(y)$.   We leave the details
to the interested reader.
\end{remark}

\subsection{Carnot groups and left translates} \label{Section:translatesproductsproofs}

Lemma \ref{lem:translatesmodulus} allows us to generalize Theorem
\ref{thm:randomtranslates} from Euclidean space to Carnot groups, as  was
discussed in Section 1.1. For the proof we also assume Theorem
\ref{thm:equaldim}, which  will be proven in Section
\ref{section:nonexpansionproofs}. We refer the reader to \cite[Chapter
11]{HK:SmetP} for definitions and background on Carnot groups in the context
of metric space analysis.

\begin{theorem}\label{thm:carnot}
Let $\mathbb{G}=(\mathbb{G},\cdot)$ be a Carnot group of homogeneous
dimension $Q>1$, equipped with its left invariant Carnot-Carath\'eodory
metric. Suppose $E\subset\mathbb{G}$ is a bounded $q$-Ahlfors regular set,
$0<q\leq Q$, and $f:\mathbb{G}\to \mathbb{G}$ is a quasiconformal mapping.
Then
  \begin{equation}
  \label{carnotconclusion}
    \dim_H f(y\cdot E) = \dim_H E,
  \end{equation}
for $\calH^Q$-a.e. $y\in\mathbb{G}$.
\end{theorem}

\begin{proof}[Proof of Theorem \ref{thm:carnot}]
We first prove the theorem in the case when $E$ is a bounded set. Since the
metric is left-translation invariant, the sets $y\cdot E$ are isometric to
$E$, and hence Ahlfors $q$-regular. Moreover, the Hausdorff measure $\calH^Q$
is positive and locally finite, and left invariant. Since Carnot groups are
unimodular (i.e., left and right Haar measures coincide), $\calH^Q$ is right
invariant as well (though see Remark \ref{rem:samenullsets} below).

Let $K\subseteq \mathbb G$ be the set of points $y\in \mathbb G$ for which
equation \eqref{carnotconclusion} fails.  Since left translations are
isometries, the measures $y_*\calH^q\lfloor_E=\calH^q\lfloor_{y\cdot E}$ are
all $q$-regular as well. By Theorem \ref{thm:equaldim}, we have
$\mod_{Q/q}(\{y\cdot E:y\in K\},\calH^Q)=0\text{.}$ Applying Lemma
\ref{lem:translatesmodulus} with $G=\mathbb G$, $\nu=\calH^Q$,
$\lambda=\calH^q\lfloor_E$, and $p=Q/q$, we have that $ \calH^Q(K)=0$ as
desired.
%
\end{proof}

\begin{remark}
\label{rem:samenullsets} In the preceding proof, we did not really need to
use the fact that $\mathbb G$ is unimodular. In any locally compact
topological group, left and right Haar measures are comparable on compact
subsets, so that right Haar measure $\nu$ is locally Ahlfors $Q$-regular.
Theorem \ref{thm:equaldim} easily generalizes to allow replacement of
$\calH^Q$ with the locally $Q$-regular measure $\nu$, so that
\[
\mod_{Q/q}(\{y\cdot E:y\in K\},\nu)=0\text{.}
\]
Lemma \ref{lem:translatesmodulus} implies $\nu(K)=0$ as before, so using
again the fact that $\nu$ and $\calH^Q$ are locally comparable, one has
$\calH^Q(K)=0$ as well.
\end{remark}





\section{General versions of non-expansion}\label{Section:results}


We begin with our most general (and technical) dimension distortion theorem,
from which all of our upper bounds on dimension distortion derive. We recall
that $\calL_d(X)$ denotes the family of lower $d$-regular measures in $X$,
and that we
 denote the support of a measure $\lambda\in \calL_d(X)$ by $E_\lambda$.  
\begin{theorem}\label{theorem:expansion}
Let $D>d>0$, and $D'>d'>0$, with $\frac{D}{d}\geq \frac{D'}{d'}$. Suppose
that $\mu$ is an upper $D$-regular measure on a separable metric space $X$,
and that $f\colon X\rightarrow Y$ is a quasisymmetric embedding.
\begin{enumerate}
\item \label{expansion2} If $\calH^{D'}\lfloor_{f(X)}$ is locally finite,  then %
for $\m_{{D}/{d}}$-almost every $\lambda\in \calL_d(X)$,
$\calH^{d'}\lfloor_{f(E_\lambda)}$ is locally finite.
\item  \label{expansion3} If $\calH^{D'}(f(X))=0$, then for
    $\m_{{D}/{d}}$-almost every $\lambda\in \calL_d(X)$,
    $\calH^{d'}(f(E_\lambda))=0$.
\end{enumerate}
\end{theorem}

By fixing the values of $D$ and $d$ and varying $D'$ and $d'$, we obtain the following
corollary.

\begin{corollary}\label{cor:expansion}
Let $D>d>0$, and let $\mu$, $X$, and $f$ satisfy the assumptions of Theorem
\ref{theorem:expansion}.  Then for $\m_{\frac{D}{d}}$-almost every
$\lambda\in \calL_d(X)$,
\begin{equation}\label{inequality:expansion}
   \frac{\dim_H f(E_\lambda)}{\dim_H f(X)}\leq \frac{d}{D}.
\end{equation}
	In particular, if  $\m_{{D}/{d}}(\calL_d(X))>0$ as well, then there is a measure $\lambda\in \calL_d(X)$ satisfying
\eqref{inequality:expansion}.
\end{corollary}


Theorem \ref{theorem:expansion} and Corollary \ref{cor:expansion} will be
proven in Section \ref{section:nonexpansionproofs}. Readers interested in
analysis on metric spaces, particularly in terms of Newton-Sobolev theory,
may wish to keep in mind the special case of curve modulus.  In this setting,
integration with respect to arc-length along a rectifiable curve
$\gamma\colon [0,l]\rightarrow $ (parametrized by arc-length) is the same as
integration with respect to the push-forward of Lebesque measure,
$\gamma_{*}(\calH_1)$, which is easily seen to be a lower $1$-regular
measure. Since almost every curve is locally rectifiable, we may apply
Corollary \ref{cor:expansion} to curve families.

\begin{corollary}\label{cor:curvetheorem}
Suppose that $\mu$ is an upper $D$-regular measure, $D>1$, on a separable
metric space $X$, and that $f\colon X\rightarrow Y$ is a quasisymmetric
embedding.  Then for $\m_D$-almost every curve $\gamma$ in $X$,
\begin{equation}\label{inequality:expansion0}
   \dim_H f(\gamma)\leq \frac{\dim_H f(X)}{D}.
\end{equation}
\end{corollary}

 \begin{proof}
    For each locally rectifiable curve $\gamma$, the corresponding
    arclength measure $\lambda$ is lower-regular and $\gamma= E_\lambda$.
    Moreover, in this case $d=1$, so inequality (\ref{inequality:expansion})
    of Corollary \ref{cor:expansion} implies
    (\ref{inequality:expansion0}).
    \end{proof}
If $X$ is Ahlfors $D$-regular, we may apply Corollary \ref{cor:expansion} to
the family $\mathcal{A}_d(X)$ of bounded \emph{Ahlfors $d$-regular} subsets
of $X$, by letting $\mu=\calH^D$ and observing that $\{\calH^d
\lfloor_E\}_{E\in\mathcal{A}_d(X)}\subseteq\calL_d(X)$.  Note that in this
case $\dim_H E=d$ and $\dim_H X=D$, and recall from the previous section that
in this context, the notion of ``almost every $d$-regular set'' is
well-defined.

\begin{corollary}\label{thm:compression}
Let $D>d>0$, let $X$ be Ahlfors $D$-regular, and let $f\colon X\rightarrow Y$ be a quasisymmetric mapping. 
Then for $\m_{{D}/{d}}$-almost every $S\in\mathcal{A}_d(X)$,
\begin{equation}\label{inequality:dregexp0}
   \frac{\dim_H f(S)}{\dim_H S}\leq \frac{\dim_H f(X)}{\dim_H X}.
  \end{equation}
	In particular, if $Y$ is also $D$-dimensional, then $\m_{{D}/{d}}$-almost every $S\in \mathcal{A}_d(X)$ satisfies
\begin{equation}\label{inequality:dregexp1}
   \dim_H f(S)\leq \dim_H S.
  \end{equation}	
\end{corollary}

    \begin{proof}
     For each $S$, let $\lambda$ be $d$-dimensional Hausdorff
     measure restricted to $S$, so $S = S_\lambda$. Then Corollary \ref{cor:expansion}
     immediately implies (\ref{inequality:dregexp0}).
    \end{proof}

\begin{remark}
Inequality (\ref{inequality:dregexp0}) may be thought of as a generalization
of the fiber-wise expansion estimate (1.18) for products. Indeed, if both
$(E,\la)$ and $(F,\nu)$ are Ahlfors regular spaces then $(X,\mu) = (E \times
F,\la\times\nu)$ is also Ahlfors regular and we may apply Corollary
\ref{thm:compression} with $S=E\times\{y\}$. Inequality
(\ref{inequality:dregexp0}) then will imply that for $\m_{D/d}$-almost every
$E\times\{y\}$, or more precisely for $\m_{D/d}$-almost every $\la_y$ like in
Section \ref{subsection:products}, the following holds
\begin{align*}
\frac{\dim_H f(E\times\{y\})}{\dim_H E\times\{y\}}
             \leq \frac {\dim_H f(E \times F)}{\dim_H E \times F},
\end{align*}
where as usual $D$ is the dimension of $X=E\times F$.
%
Moreover, by Lemma \ref{lemma:product modulus}, if $F'\subset F$ then the
family $\{\la_y\}_{y\in F'}$ has positive modulus if and only if $\nu(F')>0$.
Therefore we obtain the following strengthening of (\ref{ratios1}):
\begin{align}
\esssup_{y\in F} \frac{\dim_H f(E\times\{y\} )}{\dim_H E\times\{y\}}
             \leq \frac {\dim_H f(E \times F)}{\dim_H E \times F},
\end{align}
where $\esssup$ is taken with respect to the measure $\nu$ on $F$. Thus, we
obtain the following generalized principle of {\emph{``fiberwise expansion
implies global expansion"}}: if there is a set $F'\subset F$ such that
$\nu(F')>0$ and the fibers $E \times \{y\},y\in F'$ have their dimensions
increased by $f$ by a factor $\alpha \geq 1$, then the dimension of the whole
product $E\times F$ increases by at least a factor of $\alpha$ as well.
\end{remark}


 \begin{remark}\label{relaxationremark}
As mentioned before, an important part of Theorem \ref{theorem:expansion} is
the relaxation of the regularity conditions on the underlying space $X$ as
well as the measure $\mu$. Most significantly, \emph{we do not assume that
$\mu$ is a doubling measure}, i.e. $\mu(B(x,2r))\leq C \mu(B(x,r))$ for all
$x\in X$ and $r>0$. Instead, we  assume only upper regularity of $\mu$. This
relaxation is of paramount importance to the proof of Theorem
\ref{expandsmetricquantitative}, as Frostman's Lemma only gives us upper
regularity (see Remark \ref{doublingmetricvsmeasure} below for further
discussion of this.) As a consequence, we cannot use the well known
``Bojarsky Lemma", which is usually used in similar situations when
estimating the modulus from above, see e.g. \cite[Theorem 15.10]{H} or
\cite[Proposition 2.9]{HK}.  Instead, our argument is more in the spirit of
the proof of \cite[Theorem 1.2]{Williams}, in that a supremum must be used in
place of a summation when constructing admissible functions, see
(\ref{admissiblefunction}) below.  This method, in turn, relies on the
quasi-preservation of annuli guaranteed by Lemma \ref{lemma:qsdistortion},
and so our applications to quasiconformal maps in the plane depend heavily on
the equivalence between quasiconformality and quasisymmetry.
\end{remark}

\begin{remark}\label{doublingmetricvsmeasure}
It is important to keep in mind the distinction between a doubling metric
space and a doubling measure on a metric space.  It is easy to show that a
metric space with nonzero doubling measure must itself be doubling. Volberg
and Konyagin\cite{VolbergKonyagin} proved that, conversely, compact doubling
metric spaces admit doubling measures, and Luukainen and
Saksman\cite{LuukkainenSaksman} extended this result to arbitrary complete
metric spaces.

On the one hand, in order to invoke Frostman's Lemma in the first place, $F$ must be Borel in its completion, and must be doubling as a metric space. 
Even so, the conclusion of Frostman's Lemma only gives upper regularity for
$\mu$, and so even though a doubling measure exists, we cannot assume that
$\mu$ simultaneously has the doubling property and the desired regularity.
\end{remark}


\section{Proofs of Theorem \ref{theorem:expansion},
Corollary \ref{cor:expansion} and Theorem
\ref{thm:equaldim}.}\label{section:nonexpansionproofs}
%
%
\begin{proof}[Proof of Theorem \ref{theorem:expansion}]
We first observe that if the conclusions of the theorem hold for $d'$, then
they hold for every number larger than $d'$ as well, and so we may assume
with no loss of generality that $\frac{D}{d}=\frac{D'}{d'}$.

To begin, we suppose $U\subseteq X$ is bounded.  Then $f$ and $f^{-1}$ are
uniformly continuous, and we may let $\omega(t)$ be a modulus of continuity
for $f^{-1}$.

Fix $\epsilon>0$, and let $\delta=\omega(\epsilon)$.  Since
$\calH^{D'}\lfloor_{f(X)}$ is locally finite, we may choose balls
$$B_i'=B(y_i,s_i)\subseteq f(X)\subseteq Y$$ such that $\bigcup_{i=1}^\infty
B_i'\supseteq f(U)$, each $s_i<\epsilon$, and $\sum_{i=1}^\infty s_i^{D'}
\leq \calH^{D'}(f(U))+\epsilon$.  By Lemma \ref{lemma:qsdistortion}, there is
at each point $x_i=f^{-1}(y_i)$ a radius $r_i<\delta$ such that the balls
$B_i=B(x_i,r_i)$ satisfy
\[
B_i'\subseteq f(B_i)\subseteq f(10B_i) \subseteq \eta(10)B_i'\text{,}
\]
where $\eta$ is the distortion function for $f$.

Now, let
\begin{equation}\label{admissiblefunction}
  g_{\delta}(x)=\sup_{i\in\mathbb N} \frac{s_i^{d'}}{r_i^d}\chi_{2B_i}(x)\text{,}
\end{equation}
and for each $M>0$, define the family $\calL^{M}_{U,\epsilon}$ of measures
$\la\in\calL_d$, whose supports $E_\la$ are distorted significantly, as
follows

\begin{align}
\calL^{M}_{U,\epsilon} : = \left\{ \lambda\in\calL_d(X) \, | \,
r_\lambda>2\delta \mbox{ and } \calH^{d'}_{\epsilon\eta(10)}(f(E_\lambda\cap
U))>M C_{\lambda}\right\},
\end{align}
where $C_\lambda$ and $r_\lambda$ are the constants in the definition of
lower $d$-regularity  \eqref{ineq:p-regular}.

Next, to estimate $\int g_{\delta} \, d\la$ from below we let
$I_\lambda\subseteq \mathbb{N}$ be the set of indices $i$ such that $B_i\cap
E_\lambda\neq \emptyset$. By the basic covering lemma, there is a subset
$J_\lambda\subseteq I_\lambda$ such that
$$\bigcup_{j\in
J_\lambda}10B_j\supseteq \bigcup_{i\in I_\lambda}2B_i\supseteq E_\lambda\cap U,$$
and for $j_1\neq j_2$, $2B_{j_1}\cap 2B_{j_2}=\emptyset$.

Since the balls $2B_j$ are disjoint, we have
\begin{eqnarray*}
\int  g_\delta\,d\lambda
&\geq&
\int \sup_{j\in J_\lambda} \frac{s_j^{d'}}{r_j^d}\chi_{2B_j} \,d\lambda
= \int \sum_{j\in J_\lambda} \frac{s_j^{d'}}{r_j^d}\chi_{2B_j} \,d\lambda\\
&=&
\sum_{j\in J_\lambda} \frac{s_j^{d'}}{r_j^d}\lambda(2B_j)
\geq \frac{2^d}{C_\lambda} \sum_{j\in J_\lambda} s_j^{d'},
\end{eqnarray*}
where the last inequality holds because $\la\in\calL_d$.
Now, since $s_j<\eps$ and
$$\bigcup_{j\in J_\lambda}\eta(10)B_j'\supseteq \bigcup_{j\in J_\lambda} f(10B_j)\supseteq f(E_\lambda\cap U),$$
from the definition of $\calL_{U,\eps}^M$ we obtain
\begin{align*}
\sum_{j\in J_\lambda} (\eta(10)s_j)^{d'} \geq&
 \calH^{d'}_{\eps \eta(10)}(f(E_\lambda\cap U)) \geq M C_{\la}.
\end{align*}
Combining the last two estimates we obtain
\begin{align}
\int  g_\delta\,d\lambda
&\geq \frac{2^d}{C_\lambda} \sum_{j\in J_\lambda} s_j^{d'}
\geq \frac{2^d}{C_\lambda} \cdot \frac{M C_{\la}}{\eta(10)^{d'}}
= \frac{2^d M}{\eta(10)^{d'}}.
\end{align}
Therefore, if $C=\frac{\eta(10)^{d'}}{2^d M}$ then $C g_{\delta}$ is
admissible for $\calL^{M}_{U,\epsilon}$. Since $\mu$ is upper $D$-regular, we
can estimate the modulus of this family as follows:
\begin{align}\label{estimate2}
 \begin{split}
\m_{{D}/{d}}(\calL_{U,\epsilon}^{M})
 &\leq
C^{{{D}/{d}}} \int g_\delta^{{D}/{d}} \,d\mu
= C^{{{D}/{d}}}\int \sup_{i\in \mathbb N} \left(\frac{s_i^{d'}}{r_i^d}\chi_{2B_i}\right)^{{D}/{d}} \,d\mu\\
&\leq
C^{{D}/{d}} \int \sum_{i=1}^{\infty} \left(\frac{s_i^{d'}}{r_i^d}\chi_{2B_i}\right)^{{D}/{d}}\,d\mu
\leq
C^{{D}/{d}
}\sum_{i=1}^{\infty} \frac{s_i^{D'}}{r_i^D}\mu(2B_i) \\
&\leq
{C^{{D}/{d}}} {2^{D} C_\mu} \sum_{i=1}^{\infty} s_i^{D'}
\leq  {C^{{D}/{d}}} {2^{D} C_\mu} \left(\calH^{D'}_{\epsilon \eta(10)}(f(U))+\epsilon \right).
 \end{split}
\end{align}
We would like to let $\epsilon$ approach $0$ in (\ref{estimate2}), so we have an
estimate in terms of $\calH^{D'}(f(U))$ rather than $\calH^{D'}_{\eps
\eta(10)} f(U)$. For that note, that if $\epsilon_1>\epsilon_2$ then
$\calL^{M}_{U,\epsilon_1}\subseteq \calL^{M}_{U,\epsilon_2}$. Thus, if we
define
$$\calL_U^M=\bigcup_{n=1}^{\infty}
\calL_{U,\frac{1}{n}}^M,$$
then $\calL_U^M$ consists precisely of those measures $\lambda\in \calL_d(X)$
for which
$$\calH^{d'}(f(E_\lambda\cap
U))>M C_{\la}.$$
By Ziemer's Lemma (see Lemma \ref{lemma:modulus-properties}) we have
$$\m_{{D}/{d}}(\calL_U^M)\leq \lim_{n\to\infty} \m_{{D}/{d}}(\calL_{U,\frac{1}{n}}^M),$$
combining which with (\ref{estimate2}) we obtain the key modulus estimate
\begin{equation}
\label{keyestimate}
\m_{{D}/{d}}(\calL_U^M) \leq \eta(10)^{\frac{Dd'}{d}} C_{\mu} \frac{\calH^{D'}(f(U))}{M^{{D}/{d}}}.
\end{equation}

To prove $(1)$, note that if we define
$\calL_U^\infty=\bigcap_{k=1}^\infty\calL_U^k$ then $\calL_U^\infty$ consists
of all the measures $\lambda\in \calL_d(X)$ for which
$\calH^{d'}(f(E_\lambda\cap U))=\infty.$ Therefore, from the monotonicity of
modulus and inequality \eqref{keyestimate} it follows that for every
$k\in\mathbb{N}$ we have
$$\m_{{D}/{d}}(\calL_U^\infty)\leq \m_{{D}/{d}}(\calL_U^k)\leq\eta(10)^{\frac{Dd'}{d}} C_{\mu}\frac{\calH^{D'}(f(U))}{k^{D/d}}.$$
In particular, if $\calH^{D'}(f(U))<\infty$ then
\begin{align}\label{equality:mod}
\m_{{D}/{d}}(\calL_U^\infty)=\m_{{D}/{d}}\{\la\in\calL_d(X) \,| \, \calH^{d'}(f(E_{\la}\cap U))=\infty \} =0
\end{align} Finally, by countable subadditivity
of modulus and of Hausdorff measure, along with the separability of $X$, we
obtain \eqref{expansion2} from (\ref{equality:mod}).

To prove $(2)$ let $\calL_U^0=\bigcup_{k=1}^\infty\calL_U^{{1}/{k}}$. Note
that $\calL_U^0$ consists of those measures $\la\in\calL_d(X)$ for which
$\calH^{d'}(f(E_\lambda\cap U))>0.$ If $\calH^{D'}(f(U))=0$ then by
\eqref{keyestimate} $\m(\calL_U^{1/k})=0$ for every $k\in\mathbb{N}$ and
therefore by the countable subadditivity of modulus we obtain that
\begin{align}\label{equality:mod1}
\m_{{D}/{d}}(\calL_U^0)=\m_{{D}/{d}}\{\la\in\calL_d (X) \,| \, \calH^{d'}(f(E_{\la}\cap U))>0 \} =0.
\end{align}
Finally, again we obtain \eqref{expansion3} from (\ref{equality:mod1}) by
using the countable subadditivity of modulus, of Hausdorff measure and the
separability of $X$.
\end{proof}

\begin{remark}
As noted before, one of the most important features of the proof of Theorem
\ref{theorem:expansion} is the construction of the admissible function
$g_{\delta}$ by the formula (\ref{admissiblefunction}), which allows us to
prove the key inequality (\ref{keyestimate}) without assuming that $\mu$ is a
doubling measure. This is similar to the arguments of Williams in
\cite{Williams}.
\end{remark}

\begin{proof}[Proof of Corollary \ref{cor:expansion}]
To prove inequality \eqref{inequality:expansion}, we suppose that
$D'>\dim_H(f(X))$, and that $\frac{D}{d}=\frac{D'}{d'}$.  Then by part
\eqref{expansion3} of Theorem \ref{theorem:expansion}, we know that for
$\m_{{D}/{d}}$-almost every $\lambda\in\calL_E$, $\calH^{d'}(f(E_{\la}))=0$,
and so $\dim_H(f(E_{\la}))\leq d'$, whereby
\[
\frac{\dim_H(f(E_{\la}))}{\dim_H(f(X))}\leq \frac{d'}{\dim_H(f(X))}
\]
Since this holds for $D'=\dim_H(f(X))+\frac{1}{n}$ for each $n$, we obtain by
countable subadditivity that for $\m_{{D}/{d}}$-almost every $\lambda\in
\calL_d(X)$,
\[
\frac{\dim_H(f(E_{\la}))}{\dim_H(f(X))}\leq \frac{d'}{D'}=\frac{d}{D}\text{.}\qedhere
\]

\end{proof}



\begin{proof}[Proof of Theorem \ref{thm:equaldim}]
In light of Theorem \ref{theorem:expansion}, we only need to show that $\dim
f(E)\geq d$ for $\m_{{D}/{d}}$-almost every $E\in\calA_d(X)$.  We shall
actually show more, namely, that  $\calH^d(f(E))>0$ for almost every $E$.

When $d=D$ the theorem follows immediately from condition $N^{-1}$. Suppose
then that $d<D$. Let $L_{f^{-1}}(y,r)=\sup_{y'\in
B(y,r)}d(f^{-1}(y'),f^{-1}(y))$, and $L_{f^{-1}}(y)=\limsup_{r\rightarrow
0}\frac{L_{f^{-1}}(y,r)}{r}$. Quasisymmetry implies that
$L_{f^{-1}}(y)^D\lesssim J_{f^{-1}}(y)<\infty$, at almost every $y\in Y$.
Here
\[
J_{f^{-1}}(y):=\limsup_{r\rightarrow 0}\frac{\calH^D(f^{-1}(B(y,r)))}{\calH^D(B(y,r))}=\frac{df_{*}\calH^D\lfloor_{X}}{d\calH^D\lfloor_{Y}}
\]
is the volume derivative of $f^{-1}$.

Condition $N^{-1}$ implies that at almost every $x\in X$, we have
$L_{f^{-1}}(f(x))<\infty$.  Egorov's Theorem then gives sets $A_\epsilon$,
with $\calH^{D}(X\backslash A_\epsilon)<\epsilon$, on which
$\frac{L_{f^{-1}}(f(x),r)}{r}$ is uniformly bounded for all $x\in
A_{\epsilon}$ and $r<R_\epsilon$.  It follows that the restriction
$f^{-1}|_{f(A_\epsilon)}$ is locally Lipschitz.

Now let $A=\bigcup_{\epsilon>0} A_\epsilon$. Since $\calH^D(X\backslash
A)=0$, ${D}/{d}$-almost every measure $\lambda$ satisfies
$\lambda(X\backslash A)=0$ (this is easy to see by taking the admissible
function $\rho=\infty\cdot\chi_{X\backslash A}$ for the exceptional family).
It follows that ${D}/{d}$-almost every $E\in\calA_d(X)$ satisfies
$\calH^d(E\cap A_\epsilon)>0$ for some $\epsilon$.  Since
$f^{-1}|_{f(A_\epsilon)}$ is locally Lipschitz, we have
$\calH^d(f(E))\geq\calH^d(f(E\cap A))>0$, and the theorem is proved.
\end{proof}

\section{Products and the proof of Theorem
         \ref{expandsmetricquantitative}}\label{Section:products}%
%
%
%
 To apply Frostman's Lemma in the proof of Theorem
\ref{expandsmetricquantitative}, we need the next lemma, which follows
quickly from a similar result in \cite{Tyson:frequency}, see  Lemma 3.1 in
that paper. Though the statement there is restricted to maps from Euclidean
spaces, the proof uses no metric properties of the domain,  employing only
the fact that $\mathbb R^{m+n}=\mathbb R^m\times\mathbb R^n$ is equipped with
the product topology. Hence it applies in our setting as well.  We give the
argument from \cite{Tyson:frequency} here for the reader's convenience.
\begin{lemma}
\label{borel} Let $E$ and $F$ be topological spaces, with $E$
$\sigma$-compact, let $Z$ be a metric space, and let $f\colon E\times
F\rightarrow Z$ be continuous.  Then for each $s\geq 0$, the set
\[
F^s=\{y\in F: \dim_H(f(E\times\{y\}))>s\}
\]
is a Borel set.
\end{lemma}

\begin{proof}
Suppose first that $E$ is compact.  
It suffices to show that for each $t\geq 0$ and $\epsilon>0$ the set
\[
F^t_{\infty,\epsilon}=\{y\in :\calH^t_\infty(f(E\times\{y\}))\geq\epsilon \}
\]
is closed, since $F^s=
\bigcup_{t>s}\bigcup_{\epsilon>0}F^t_{\infty,\epsilon}$. To this end, let
$y\in F\backslash F_{\infty,\epsilon}^t$.  Then there is a sequence of open
balls $B(z_i,r_i)$, with each $z_i\in Z$,  and with $\sum_{i=1}^\infty
r_i^s<\epsilon$, such that
\[
f(E\times\{y\})\subset\bigcup_{i=1}^{\infty}B(z_i, r_i)\text{,}
\]
or equivalently,
\[
E\times\{y\}\subseteq\bigcup_{i=1}^{\infty}f^{-1}(B(z_i, r_i))\text{.}
\]
By the continuity of $f$, $\bigcup_{i=1}^{\infty}f^{-1}(B(z_i, r_i))$ is
open, so by the compactness of $E$, there is an open set $U\ni y$ such that
\[
E\times U\subseteq\bigcup_{i=1}^{\infty}f^{-1}(B(z_i, r_i))\text{,}
\]
so that $U\subseteq F\backslash F^t_{\infty,\epsilon}$. Thus $F\backslash
F^t_{\infty,\epsilon}$ is open, whereby $F^t_{\infty,\epsilon}$ is closed as
desired.

Finally, suppose $E=\bigcup_{i=1}^\infty E_i$, with each $E_i$ compact.
Then $ F^s=\bigcup_{i=1}^\infty F^s_i\text{,} $ where
\[
F^{s}_i=\{y\in F: \dim_H(f(E_i\times\{y\}))>s\}\text{.}
\]
Since these sets are Borel by the compact case of the lemma, $F^s$ is Borel
as well.
\end{proof}

\begin{proof}[Proof of Theorem \ref{expandsmetricquantitative}]
The theorem is trivial if $d'=D'$.  We also observe that since quasisymmetric
maps are uniformly continuous on bounded sets, they extend to the completions
of the spaces on which they are defined.  We may therefore assume that
$d'<D'$ and that $E$ and $F$ are complete.  Note that $d$-regularity implies
the doubling property, so that $E$ is a complete doubling metric space.  As
such, $E$ is proper (balls are compact), and a fortiori $\sigma$-compact.

Let $D=\frac{d}{d'}D'$, and $F_{d'}=\{y\in F:\dim_H f(E\times\{y\})>d'\}$.
Suppose by way of contradiction that
\begin{equation}
  \calH^{\frac{d}{d'}D'-d}(F_{d'})>0\text{.}
\end{equation}

By Lemma \ref{borel}, $F_{d'}$ is Borel set. Thus it follows from Frostman's
Lemma that there is a nonzero upper $(D-d)$-regular measure $\nu$ on
$F_{d'}$.

Since $E$ (and hence $\calH^d$) is $d$-regular, 
the measure $\mu=\calH^d\times\nu$ is upper $D$-regular on $X=E\times
F_{d'}$, and so by Lemma \ref{lemma:product modulus}, we have that
$\m_{{D}/{d}}\left(\{\calH^d\lfloor_{E\times\{y\}}:y\in F_{d'}\}\right)>0$.
From this and Corollary \ref{cor:expansion} (which applies by the lower
regularity of each measure $\calH^d\lfloor_{E\times\{y\}}$), we then conclude
that there is some $y\in F_{d'}$ such that
\[
\dim_H f(E\times \{y\}) \leq \frac{d}{D}D'= d'\text{,}
\]
contradicting the definition of $F_{d'}$.
\end{proof}

\begin{proof}[Proof of Corollary \ref{expandsmetric}]
Let $d'<\inf_{y\in F}\dim_H f(E\times\{y\})$, and $D'=\dim_H(f(E\times F))$.
Using the notation of the preceding proof, we have that $F=F_{d'}$, and so by
Theorem \ref{expandsmetricquantitative}, we obtain
\[
\calH^{\frac{d}{d'}D'-d}(F)=0\text{,}
\]
so that $\dim_H(F)\leq \frac{d}{d'}D'-d$.

Since $E$ is Ahlfors $d$-regular, the packing dimension of $E$ is equal to
its Hausdorff dimension (see Theorem $6.13$ of \cite{Mattila}). Therefore we
also have (see e.g.\ Corollary $8.11$ of \cite{Mattila}) that
\[
\dim_H (E\times F) = \dim_H E +\dim_H F \leq \frac{d}{d'}D'\text{,}
\]
whereby
\[
d'\leq\frac{dD'}{\dim_H(E\times F)}=\frac{\dim_H(f(E\times F))}{\dim_H(E\times F)}\cdot\dim_H(E)\text{.}
\]
Since this holds for every $d'<\inf_{y\in F}\dim_H f(E\times\{y\})$, the
proof is complete.
\end{proof}



\section{Proofs of Theorem \ref{first thm} and Corollary
         \ref{cor:sharp-dimension} }\label{Section:sharpnessproofs}
\begin{proof}[Proof of Theorem \ref{first thm}]

%
%

For $0<t<1$, let $\alpha= 2^{-1/t}$, and let $S$  be the Cantor set given  by a standard
iterative construction that starts with   $I_0 = [0,1]$ and
 replaces  each $n$th generation interval $I$ with two $(n+1)$st
generation intervals of length $\alpha |I|$ and distance  $\frac 13(1-2
\alpha)|I|$ from each other and from the endpoints of $I$. It is easy to
check that $ \dim(S) = t$.
The proof is somewhat cleaner when we restrict to the case that $\alpha^{-k}$
is an integer  for some  positive integer $k$ (which may depend on $\alpha$).
Note that this holds whenever $t\in \mathbb Q$.  The proof in the general
case follows the same idea, with a few technical modifications, described in
the ensuing Remark \ref{reductionremark}.
Taking $\alpha^{-k}=N$, we have $N^{t}=2^k$. By taking multiples of $k$, we
may assume $N$ is as large as we wish.

The basic building block in the
construction  of $f$ is a quasiconformal map of the unit square $Q=[0,1]^2$,
\begin{align}
\Phi :[0,1]^2 \to [0,1]^2.
\end{align}
Let $\{I_j\}$, $j =1, \dots , 2^k$ be the $k$th generation covering intervals
of $S$, each of which has length $\alpha^k=\frac{1}{N}$,   and let $R_j =
[0,1] \times I_j$ be rectangles with their short edges on the vertical sides
of $Q$. Note that each of these is isometric to the $1 \times \frac 1N$
rectangle
$$R=\{(x,y)\in \mathbb{R}^2 : 0<x<1, 0<y<1/N\}.$$
The map $\Phi$ will be conformal on each of these rectangles
and will be quasiconformal on the rest of $Q$. Our construction also gives
that $\Phi$ is the identity on the top and bottom edges of $Q$ and it is
symmetric with respect to the vertical bisector of $Q$, so $ \Phi(1,y) =
\Phi(0,y)+(1,0) $ on the vertical sides of $Q$. This means that $\Phi$ can be
extended to a quasiconformal map of the whole plane by simply mapping each
square $Q+(n,m)$ to itself by $(x,y) \to \Phi(x-n, y-m) +(n,m)$.

The map $\Phi$ is constructed by specifying a generalized quadrilateral $\mathscr{T}
\subset Q$ ($\mathscr{T}$ for ``tube'') that has two opposite sides on the vertical
 sides of $Q$ and conformal modulus $N$
(the same as  $R$). This means there is a conformal map
$$\phi :R \mapsto \mathscr{T}$$
that maps vertices to vertices. This map is used to define a conformal map of
each $R_j$ to a translate  $\mathscr{T}_j$ of $\mathscr{T}$ that connects the left and right
sides of $Q$. The tubes $\mathscr{T}_j$ will have disjoint closures that do not hit the
top and bottom edges of $Q$ and so the complement of these tubes in $Q$ are
$2^k +1$ regions. We define a quasiconformal map from each component of $Q
\setminus \cup_j R_j$ to the corresponding component of $Q \setminus \cup_j
\mathscr{T}_j$  so that it extends the mapping on each $R_j$, is the identity on the
top and bottom edges of $Q$ and is symmetric on the vertical edges of $Q$.
The quasiconformal constant $K$  of this map depends on the geometry of and
spacing between the $\mathscr{T}_j$, but is finite for the examples we will build.

The tube $\mathscr{T}$ will be constructed so that
\begin{equation}
\label{derivativebound}
|\phi'|  \geq C_1N^{\frac{1-t}{2}}\text{,}
\end{equation}
on all of $R$ and with some constant $C_1 > 0$ that is independent of $k$ and
$N$. First we use this estimate to finish the proof of the theorem, and then
we construct a tube for which this estimate is true.

Let $f_1= \Phi$. The rectangle $R$ has an obvious decomposition into $N$
squares of side length $1/N$.   Define a map $f_2:Q \to Q$ as the identity
outside $\cup_j R_j$ and inside each $R_j$ use a scaled version of $\Phi$ to
map each subsquare of $R_j$ to itself. In general, $f_n:Q \to Q$ is defined
as the identity off the $2^{nk}$ rectangles corresponding to the intervals of
generation $nk$ covering $S$ and is defined using a scaled copy of $\Phi$ on
the $N^n$ squares, of sidelength $\alpha^{nk}=N^{-n}$, making up each such
rectangle. Then
$$g_n = f_1\circ \dots \circ f_n$$
is  quasiconformal with constant $K$ (at most one map in the composition is
non-conformal  when applied to any point). Thus, the limiting map is also
$K$-quasiconformal, and we finally define $f$ as follows,
\begin{align*}
f=\lim_{n\to\infty} g_n = \lim_{n\to\infty} f_1\circ \dots \circ f_n.
\end{align*}
Moreover, on every generation $nk$ square $\tilde{Q}$ in one of the scaled
copies of $R$, $f$ restricts to a map of the square onto itself, followed by
a succession of $n$ conformal maps, each satisfying inequality
\eqref{derivativebound}.  We therefore have, for every generation $nk$ square
$\tilde{Q}$ in a scaled copy of $R$, the estimate
\begin{equation}
\label{diamestimate}
\diam(f(\tilde{Q}))\geq C_1^nN^{n(\frac{1-t}{2})}\diam(\tilde{Q}) = C_1^nN^{-n(\frac{1+t}{2})}\text{.}
\end{equation}
Since $C_1$ does not depend on our choice of $N$, we may suppose as well that
$N$ was chosen large enough so that $C_1\geq
N^{-(\frac{\epsilon}{1-\epsilon})(\frac{1+t}{2})}$, whereby

\begin{equation*}
\diam(f(\tilde{Q}))\geq N^{-n(\frac{1+t}{2(1-\epsilon)})}\text{.}
\end{equation*}

Consider a Borel subset $A\subseteq \mathbb R\times S$, and let
$\delta<\dim_H(A)$. By Frostman's Lemma, there is a positive measure $\mu$
supported on $A$ such that
$$\mu(\tilde{Q}) \leq C_{\mu} (\diam \tilde{Q})^{\delta} = C_{\mu}N^{-n\delta}.$$ Let
$\nu=f_{\#}\mu$ be the pushforward measure, and
 let $s=\frac{2\delta}{(t+1)}(1-\epsilon)$.  Then for every generation $nk$ square $\tilde{Q}$ in a scaled copy of $R$,
 \begin{align}
\label{upperregestimate}
\diam(f(\tilde{Q}))^s  \geq  N^{-ns(\frac{1+t}{2(1-\epsilon)})}=N^{-n\delta}\geq C_\mu^{-1} \mu(\tilde{Q})= C_\mu^{-1} \nu(f(\tilde{Q}))
\text{.}
\end{align}
%
%
Next, note that every ball $B=B(x,r)$ in $\mathbb R\times S$ can be covered
by a uniformly bounded number of generation $nk$ squares of comparable
diameter in the scaled copies of $R$. Indeed, choose the smallest
$n\in\mathbb{N}$ so that $N^{-n}\geq 2r$. Then $B$ intersects at most $2$
rectangles of width $N^{-k}$ and thus at most $2^k$ rectangles of width
$N^{-(k+1)}$. Since each rectangle contains $N$ squares, we can cover $B$ by
$L=2^k\cdot 2N$ squares $\tilde{Q}$ of side length $N^{-(k+1)}$. Since $
\diam (\tilde{Q}) < r$ by Lemma \ref{upgradedmasslemma}, $\dim(f(A))\geq
\frac{2\delta}{(t+1)}(1-\epsilon)$. Since this  holds for arbitrary
$\delta<\dim_H(A)$, we obtain $\dim(f(A))\geq
\frac{2\dim_H(A)}{(t+1)}(1-\epsilon)$.

We are now done, except we must build a  tube $\mathscr{T}\subset [0,1]^2$ that has the proper
estimate on the  conformal map to a rectangle.

\subsection*{The ``Tube" construction} As before, let $N = \alpha^{-k}$.
The $ 1 \times \frac 1N$ rectangle $R$ is divided into $N$ disjoint squares and
there are $2^k$ such rectangles. Since $\alpha < \frac 12$, we have $N \gg
2^k$ when $k$ is large. Let
$$ m = \left\lfloor  \sqrt{\frac {N-1}{2^{k-1} }} \right\rfloor,$$
so that
$$  \frac 12 N \leq M=  m^2  2^{k-1}  +1  \leq N, $$
if $N$ is large enough.

\begin{figure}[htb]
\centerline{
\psfig{figure=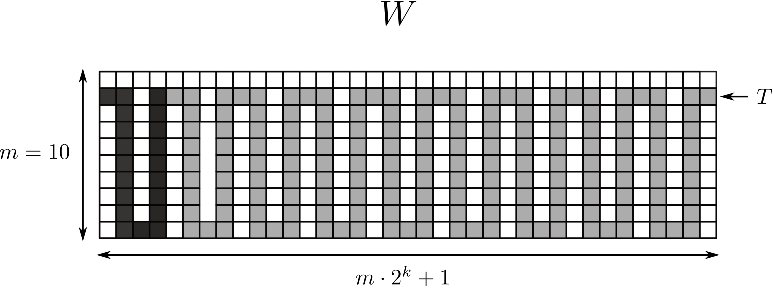,height=1.5in}
}
\caption{ \label{Wiggle521}
The first version of the ``large" tube $T\subset W$ is made up of $M = m^2 2^{k-1}+1$ shaded squares
forming a connected subset of a $m \times (2^k m +1)$ grid. Here we have
taken $k=2$ and $m=10$. The dark grey squares make up the pattern, which when repeated periodically $m2^{k-2}$ times, forms $T$, except for the ``last" or rightmost sub-square which is adjacent to the right edge of $W$.
}
\end{figure}



We start by constructing a large copy of the tube $\mathscr{T}$, which we will denote by ${T}$. Consider the $(2^km+1) \times m$ rectangle ${W}$ shown in Figure
\ref{Wiggle521}. 
The shaded (dark and light grey) unit area squares in ${W}$ form a tube ${T}$ that connects the two
vertical sides of $W$. More precisely, denoting by $Q_{p,q}$ the square $[p-1,p]\times[q-1,q]$ in the plane, the tube $T$ is obtained by considering the union of $2m$ (dark grey) squares
$$Q_{1,m-1};Q_{2,m-1},\ldots,Q_{2,1};Q_{3,1};Q_{4,1},\ldots,Q_{4,m-1},$$
repeating this pattern periodically $m\cdot 2^{k-2}$ times $(k\geq2)$ and attaching the rightmost square $Q_{m2^{k}+1,m-1}$.
Alternatively, we may write the tube $T$ as follows

\begin{align*}
  T= \bigcup_{i=0}^{m\cdot 2^{k-2}-1}
  &
  Q_{4i,m-1} \cup (Q_{4i+1,m-1}\cup\ldots\cup Q_{4i+1,1}) \cup \\
  \bigcup_{i=0}^{m\cdot 2^{k-2}-1}  & Q_{4i+2,1} \cup [ Q_{4i+3,1}\cup\ldots\cup Q_{4i+3,m-1}]
   \cup Q_{m2^{k}+1,m-1}.
\end{align*}

Let $M$ be the number of disjoint subsquares in ${T}$, which is also the area of $T$. Note, that the number of the subsquares of ${T}$, which do not intersect the right side of ${W}$ is exactly the half of the subsquares of $W$, which do not intersect its right side. Indeed, for each column of the grid containing only one square in $T$, move that square to the top row of the next column to the right. This gives a sequence of alternating ``full" and ``empty" columns, see Figure \ref{Wiggle521}. Since there are $m$ subsquares of $W$ in any column, we have
\begin{align*}
  M-1=\frac{\calH^2({W}) - m}{2} = \frac{m\times (2^k m +1) -m}{2}= m^2 \cdot 2^{k-1}.
\end{align*}
Therefore, $M=m^2 \cdot 2^{k-1} +1.$

We think of $T$ as a
generalized quadrilateral with two sides (the ``short sides'') on the
vertical sides of $W$ and two other sides (the ``long sides'') that connect
the vertical sides of $W$.

\begin{figure}
\centerline{
\psfig{figure=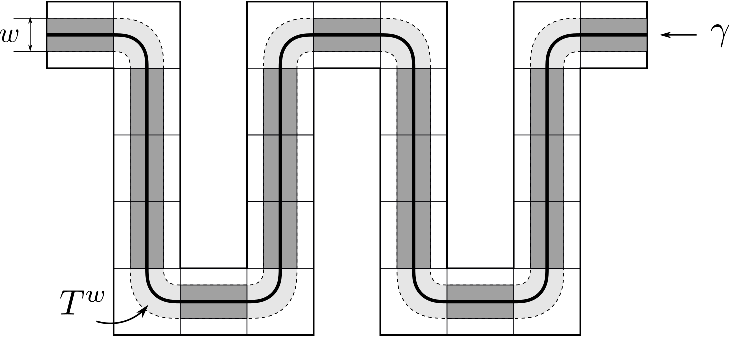,height=1.5in}
}
\caption{ \label{RoundCorners}
The tube $T^w$ has rounded corners, so that the conformal map
to the tube from the rectangle of the same modulus
has a ``nice derivative'' (the minimum and maximum expansion have bounded
ratio). Moreover, $T^w$ has a smaller width $w$; this increases the modulus of the family connecting
the two long edges of the tube. By ``thinning'' the
tube we can make its modulus exactly what we want. Since we start within
a factor of 4 of the desired modulus, only a bounded amount of thinning
is needed. The darker rectangles in ``non-corner" squares are used to estimate the modulus from below.
}
\end{figure}
Up to a similarity, the region $T$ is almost the tube we want, but it is convenient to ``round
the corners'' as shown in Figure \ref{RoundCorners}. Rounding the corners
will imply that the derivative of the conformal map of a rectangle $R$ (of
the same modulus as $T$)  to $T$ (as well as to the final tube $\mathscr{T}$) is everywhere comparable to the ratio of the
widths of $R$ and $T$, as is demonstrated in Lemma \ref{lemma:derivative}
below. Next, we carry out this procedure in more detail.

Let $\g$ be the ``core curve" which connects the midpoints of the short sides
of $T$. More precisely, if $Q_i$ is not a ``corner subsquare" of $T$ then
$\g\cap Q_i$ is a horizontal or vertical interval in $Q_i$
connecting the midpoints of opposite sides. On the other hand, if $Q_i$ is a
``corner subsquare" then $\g\cap Q_i$ is the arc of the circle of radius
$1/2$ connecting the midpoints of two adjacent sides of $Q_i$, see Figure
\ref{RoundCorners}.

For $0<w\leq 1$ we denote by $T^{w}$ the intersection of $T$ and the
$w/2$-neighborhood of $\g$. Thus $T^w$ is a tube around $\g$ of ``width" $w$.
Moreover, for $w<1$ the long sides of $T^w$ are $C^{1+\alpha}$ curves for
every $\alpha<1$, since they are differentiable curves consisting of line
segments and circular arcs. Next, we will show that $w$ can be chosen so that
$T^w$ is conformally equivalent to the $1\times 1/N$ rectangle $R$. For that
we will need the following lemma.

\begin{lemma}
For $w\in(0,1]$ let $\G^{w}$ be the path family connecting the long sides of
$T^{w}$. Then  we have the following estimates
\begin{align}\label{modest}
\frac{M  - 2^k m}{w} \leq \mod_2 \G^w \leq \frac{l(\g)}{w} < \frac{M}{w},
\end{align}
where $l(\g)$ is the length of $\g$. In particular if $2^km<M/2$ then for
$w=1$ we have
\begin{align}\label{modest1}
\frac 12 M \leq M  - 2^k m
\leq \mod_2 \G^1 < M .
\end{align}
\end{lemma}

\begin{proof}
To prove the right hand side of (\ref{modest}), take the constant metric
$\rho\equiv w^{-1}$ on $T^w$. Since every path connecting the long sides of
$T^w$ has length $\geq w$, we see that $\rho$  is admissible for $\G$. Hence,
\begin{align*}
  \m_2 \G \leq \int_T w^{-2} dxdy = \frac{\calH^2(T^w)}{ w^{2}}\leq \frac{l(\g)w}{w^2} = \frac{l(\g)}{w}.
\end{align*}
To prove the left hand inequality of (\ref{modest}) consider all the
``non-corner'' squares in $T$. Note that there are $2^km$ ``corner" squares
and therefore only $M-2^km$ ``non-corner" squares, and we denote them by
$Q_i, i=1\ldots,M-2^km$. For each such ``non-corner" square $Q_i$ let
$\G_i^w$ be the subfamily of paths in $\G^w$ contained in $Q_i$. Note, that
the paths in  $\G_i^w$ connect the opposite sides of the rectangle $Q_i\cap
T^w$ of length $1$ and therefore $\m_2 (\G_i^w) = 1/w$. Since $\G_i^w$'s are
disjoint families, we have
$$\m_2(\G^w) \geq \m_2\left(\bigcup_{i=1}^{M-2^km}\G_i^w\right)  = \sum_{i=1}^{M-2^km} \m_2(\G_i^w) = \frac{M-2^km}{w}.\qedhere$$
\end{proof}

Note, that $\m_2(\G^w)$ changes continuously with $w$ and by
(\ref{modest}) can be made as large as desired by taking $w$ small enough.
Thus, there is a $w_0>0$ such that $\m_2 (\G^{w_0}) =N$. Moreover, from
(\ref{modest}) it follows that for $w_1,w_2\in(0,1]$ we have
\begin{align*}
\frac{1}{2} \frac{\m(\G^{w_2})}{\m(\G^{w_1})} \leq \frac{w_1}{w_2} \leq 2\frac{\m(\G^{w_2})}{\m(\G^{w_1})}.
\end{align*}
In particular, from (\ref{modest1}) we obtain
\begin{align*}
\frac{1}{4}\leq   2^{-1} \frac{M}{N}\leq w_0 < 1
\end{align*}
and the width of the thinner tube is comparable to the width of the original
tube.

Slightly abusing the notation, we let $T$ denote the large rounded, thinned
tube $T^{w_0}$ of modulus $N$ and finally define the tube $\mathscr{T}$ by
\begin{align*}
  \mathscr{T} = \sigma(T)
\end{align*}
where $\sigma$ is the similarity of the plane $$\sigma(x,y)=(m2^k+1)^{-1}(x,y).$$ Therefore $\mathscr{T}$ is a copy of $T$ which connects the vertical sides of $Q$ and intersects only $M$ subsquares of $[0,1]^2$ of sidelength $(2^k m + 1)^{-1}$.

Take the unit square $Q = [0,1]^2$ and
subdivide it into $(2^k m+1)^2$ disjoint subsquares of side length $(2^k m
+1)^{-1}$, grouped into $2^k$ rectangles of dimension $(2^k m+1) \times m$
and the single $(2^k m+1)\times 1$ strip at the bottom. See Figure
\ref{Wiggle521ED}.
\begin{figure}[htb]
\centerline{
\psfig{figure=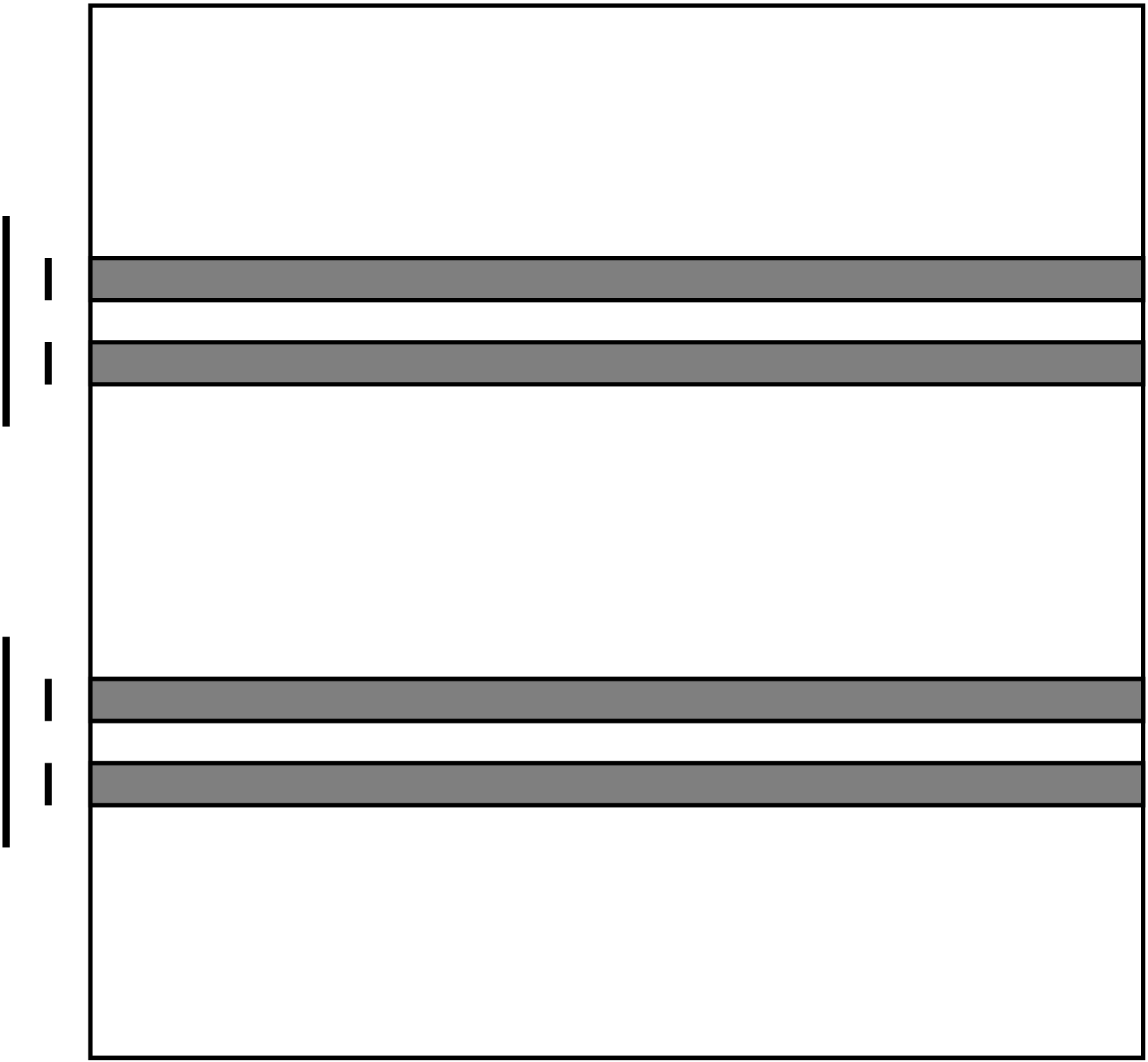,height=1.5in}
$\hphantom{xxxx}$
\psfig{figure=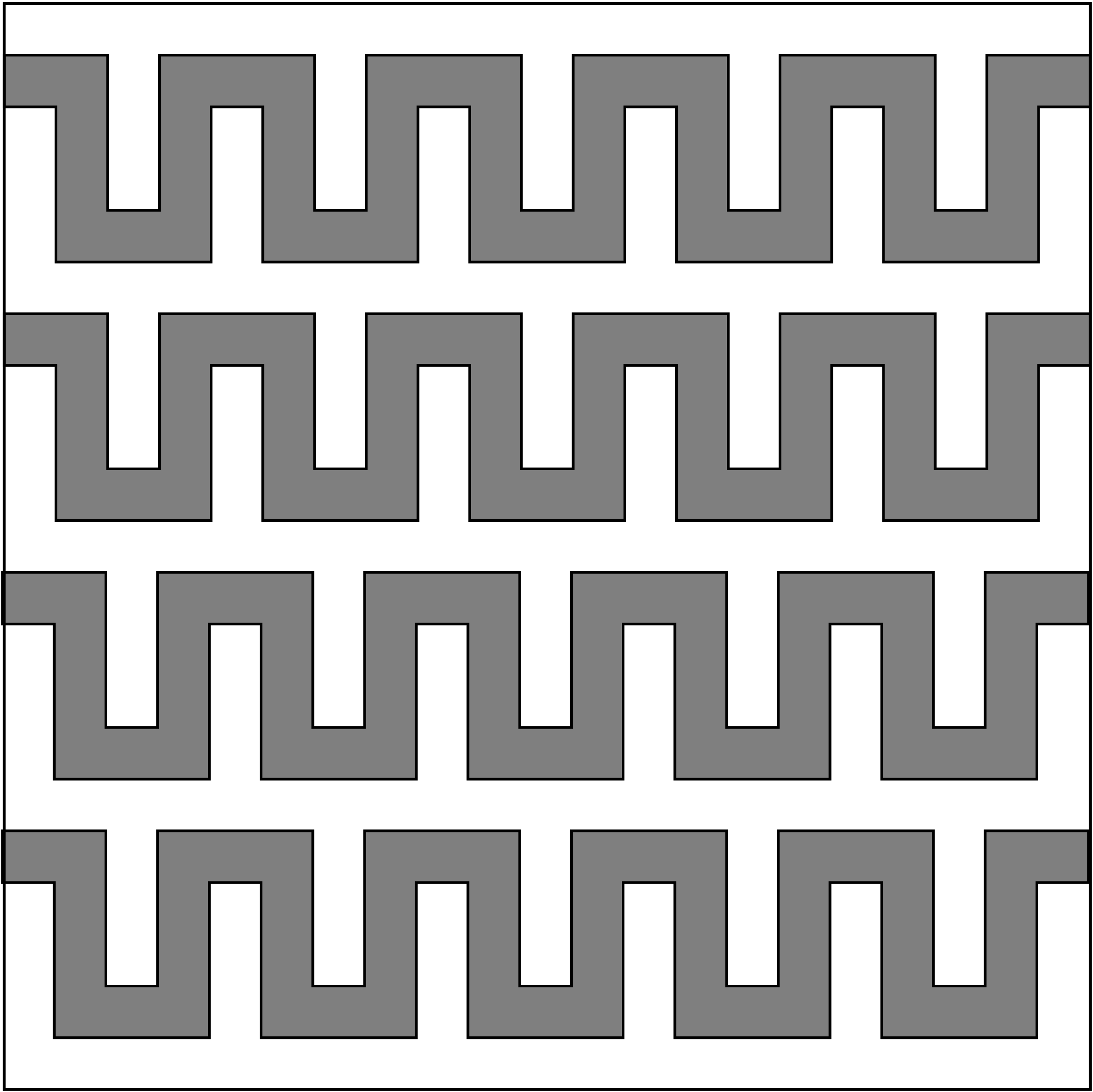,height=1.5in}
}
\caption{ \label{Wiggle521ED}
On the left is $Q$ and the shaded rectangles $R_j$. On the
right are  the tubes $T_j$ (we have omitted the rounding and
thinning to make the picture simpler).
 Each $R_j$ is conformally mapped
to the corresponding $T_j$ (they have the same modulus, so we
can send vertices to vertices) and  we quasiconformally extend
these maps to map of  $Q$ onto itself that is the identity on the top and
bottom edges and is symmetric on the left and right edges (the
tubes are symmetric with respect to the vertical bisector of
$Q$, so this is possible).
}
\end{figure}

Inside each rectangle, place a (shifted up) copy of the tube $\mathscr{T}$. These are the $\mathscr{T}_j$'s mentioned earlier. Finally, we show in Lemma \ref{lemma:derivative} that the conformal map from $R_j$ to $\mathscr{T}_j$ (mapping vertices to
vertices), or equivalently the map $\phi:R \to \mathscr{T}$ has derivative
everywhere comparable to $\frac{\text{width}(\mathscr{T})}{\text{width}(R)}$.
But, since $N^t=2^k$ and thus $m\simeq \sqrt{\frac{N}{2^k}} = N^{\frac{1-t}{2}}$ it follows, that
\begin{align*}
  |\phi'|\simeq \frac{\text{width}(\mathscr{T})}{\text{width}(R)}= \frac{w_0(2^k m+1)^{-1}}{N^{-1}} \simeq \frac{N}{2^km}
  \simeq \frac{N}{N^t N^{\frac{1-t}{2}}} = N^{\frac{1-t}{2}}.
\end{align*}
%
%
thus completing the proof of Theorem \ref{first thm} (except for the proof of Lemma \ref{lemma:derivative}).

\begin{lemma}\label{lemma:derivative}
The conformal map $\phi:R\to\mathscr{T}$ taking the vertices of $R$ to vertices of $\mathscr{T}$, has a
derivative that is everywhere comparable to the width of $\mathscr{T}$  divided by the
width of $R$, i.e.
\begin{align}\label{estimate:derivative}
  |\phi'(z)|\simeq \frac{\mathrm{width}(\mathscr{T})}{\mathrm{width}(R)}.
\end{align}
\end{lemma}

\begin{proof}
%
We will show that the absolute value of the derivative of $\phi^{-1}:\mathscr{T}\to R$ is comparable to $\frac{\mathrm{width}(R)}{\mathrm{width}(\mathscr{T})}$ everywhere in $\mathscr{T}$.
%
Using complex notation we consider the linear maps
\begin{align*}
  s_{\mathscr{T}}(z)=\sigma^{-1} z = (2^km+1)z = \frac{w_0}{\mathrm{width}(\mathscr{T})}z, \quad s_R(z)=\frac{1}{\mathrm{width}(R)}z = Nz
\end{align*}
and note that if we define
$$\psi=s_R\circ\phi^{-1}\circ s_{\mathscr{T}}^{-1},$$
then $\psi$ maps the large rounded tube $T$ onto a rectangle of width $1$ and we have the following expression for the derivative of $\phi$,
$$|(\phi^{-1})'|=|(s_R^{-1}\circ\psi\circ s_{\mathscr{T}})'|=|(s_R)'|^{-1}|\psi'||(s_{\mathscr{T}})'|=w_0\frac{\mathrm{width}(R)}{\mathrm{width}(\mathscr{T})}|\psi'|\simeq \frac{\mathrm{width}(R)}{\mathrm{width}(\mathscr{T})}|\psi'|.$$
Thus,  to obtain the estimate (\ref{estimate:derivative}) it is enough to show that if $\psi$ is the conformal mapping of $T$ onto a rectangle of width $1$, taking vertices to vertices,  then we have that $|\psi'|\simeq 1$ everywhere in $T$.

Now, the conformal  map $\psi$ from the tube $T$ to the rectangle of width $1$ can be extended by
Schwarz reflection across the ``short ends'' of the tube and rectangle  to a
map $\tilde{\psi}$ from an infinite tube $\tilde{{T}}$ of width $w_0$ to the infinite strip
$S =\{ x+iy: 0 < y <1\}$. Assume we have done this (to avoid separate arguments near the short
ends). The
Koebe $\frac 14$-theorem then implies that a conformal map
$\tilde{\psi}:\tilde{T} \to S$ has derivative
$$ |\tilde{\psi}'(z)| \simeq \frac {\dist(\tilde{\psi}(z),\partial S)}
                        {\dist(z, \partial \tilde{{T}})},$$
so it suffices to prove that the right hand side is uniformly bounded and
bounded away from zero.

Let $v$ be the imaginary part of $\tilde{\psi}$; it is a harmonic function on
the infinite tube $\tilde{{T}}$  with boundary values $v=1$
 on one boundary component (call it $X_1$) and    $v=0$ on the other
(call it $X_0$). Then $\dist(\tilde{\psi}(z),\partial S)=\min(v(z),1-v(z))$
and it suffices to prove the following implications:
\begin{align}
& \dist(z,X_0) \to 0 \quad \Longrightarrow  \quad v(z) \simeq \dist(z,X_0)  \label{imaginary-estimate1}\\
& \dist(z,X_1) \to 0 \quad \Longrightarrow \quad 1-v(z) \simeq \dist(z,X_1). \label{imaginary-estimate2}
\end{align}
We will prove (\ref{imaginary-estimate1}) in detail; the estimate (\ref{imaginary-estimate2})  follows from an essentially identical argument.

\begin{figure}[htb]
\centerline{
\psfig{figure=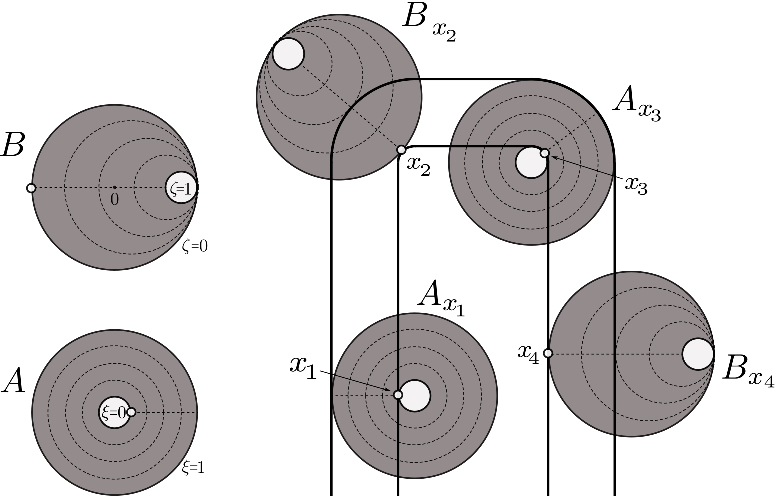,height=3in}
}
\caption{ \label{UVWcompare}
On the left are the domains of definition of $\xi$ and
$\zeta$ and their boundary values. A special boundary
 point is marked
with a white dot. Explicit computations show that $\xi$
and $\zeta$ vanish at this point and grow almost linearly on the
interior segment normal to this point. The dashed lines
show level lines of the two functions. The righthand side
of the figure shows how  these domains can be positioned
with the
special point at any boundary point of the tube; this shows
that $v$ has the correct estimate on every normal segment
crossing the tube, hence on the whole tube.
}
\end{figure}

To prove (\ref{imaginary-estimate1}) we will show that for any point $x\in X_0$ there exist functions $\zeta_x$ and $\xi_x$ defined on the segment $I_x$ which passes through $x$ and is orthogonal to $\partial T$  such that for all $z\in I_x$ we have
\begin{align}\label{est:auxiliary}
  \zeta_x(z) \leq v(z) \leq \xi_x(z),
\end{align}
and such that as  $\dist(z,X_0)\simeq 0$ the following estimates hold with constants independent of $x\in X_0$,
\begin{align}
  &\dist(z,X_0)\lesssim \zeta_x(z),\label{est:auxiliary2}\\
  &\xi_x(z) \lesssim \dist(z,X_0), \label{est:auxiliary3}
\end{align}
thus implying (\ref{imaginary-estimate1}).

The functions $\xi_x$ and $\zeta_x$ will be the ``transported versions" of harmonic functions $\xi$ and $\zeta$ which we define next, see also Figure \ref{UVWcompare}.

First, let
$A = \left\{z\in\mathbb{C} : \frac{1-w_0}{2} < |z| < \frac{1+w_0}{2} \right\},$ and define $\xi:A\to\mathbb{R}$ to be the harmonic function in $A$ with boundary values
\begin{align}
  \xi(z) = \begin{cases}
    0 \mbox{ for }  |z| = \frac{1-w_0}{2},\\
    1 \mbox{ for }  |z| = \frac{1+w_0}{2}.
  \end{cases}
\end{align}
The function $\xi$ can be explicitly written as $\xi(z) = (\log \frac{1+w_0}{1-w_0})^{-1}{\log\frac{2|z|}{1-w_0}}$. Then, as $z\in A$ approaches $\frac{1-w_0}{2}$ along the real axis, we have the estimate $\xi(z) \simeq \dist (z,\frac{1-w_0}{2})$, which follows from the fact that the partial derivative of $\xi$ at $\frac{1-w_0}{2}$ with respect the first coordinate is positive (this is clear from the construction, but could also be checked using the formula for $\xi(z)$ given above).


Now, for a fixed $x\in X_0$ there is a unique isometry $r_x$ of the plane such that $r_x(\frac{1-w_0}{2})=x$ (i.e. the special point of $A$ is mapped to $x\in X_0$), and the circle $r_x(\{|z|=\frac{1-w_0}{2}\})$ is tangent to $X_0$ at $x\in X_0$ and is located outside of
the (open) tube $\tilde{T}$. We denote $A_x=r_x(A)$ and define the harmonic function $\xi_x$ on $A_x$ by  $\xi_x=\xi\circ r_x^{-1}$. Then
\begin{align*}
v \leq  1 = \xi_x &\mbox{ on }  \partial A_x \cap \tilde{T} \\
v = 0 \leq \xi_x &\mbox{ on } \partial \tilde{T}\cap A_x,
\end{align*}
therefore $v \leq\xi_x$ on the whole boundary
of $A_x\cap \tilde{T}$. By the maximum principle,
$v\leq \xi_x$ on $A_x \cap \tilde{T}$, which give one side of (\ref{est:auxiliary}).

 From the behaviour of $\xi$ near the special point $\frac{1-w_0}{2}=r_x^{-1}(x)$ of $A$, we have that as $z$ approaches $x$ along the normal segment to $X_0$ the following estimates hold
%
%
$$\xi_x(z)=\xi(r_x^{-1}(z)) \simeq \dist(r_x^{-1}(z),r_x^{-1}(x)) = \dist(z,x) = \dist(z,X_0).$$
%
%
Therefore (\ref{est:auxiliary2}) holds as $z$ approaches $\partial X$ along the normal segment $I_x$ and the constants are independent of $x$, since they depend only on $\partial_1\xi(\frac{1-w_0}{2})$ (here $\partial_1$ denotes the partial derivative in the first coordinate).

To obtain the left hand inequality in (\ref{est:auxiliary}) we define the second function $\zeta$ on the cusped region $B=\{|z|<\frac{1+w_0}{2}\}\setminus D_0$ where
$D_0$ is the (closed) disk of radius $(1-w_0)/2$ that is tangent to the outer circle at the point $\frac{1+w_0}{2}\in\mathbb{C}$ and is contained in $\{|z|<\frac{1+w_0}{2}\}$. We set $\zeta$ to be the harmonic function on $B$ such that $\zeta=0$ on the boundary
of the larger disk and $\zeta=1$ on the boundary of the smaller disk. The
dashed curves are circles and show the level lines of $\zeta$. In Figure \ref{UVWcompare}, the special boundary point $-\frac{1+w_0}{2}\in\partial B$ is marked as a white dot; it is opposite to the point where the two boundary circles are
tangent.

The function $\zeta$ can be
computed explicitly by mapping the cusp region to the infinite strip $\{0<y<1\}$ by a M{\"o}bius
transformation and the estimate $\partial_1\zeta(-\frac{1+w_0}{2})>0$ clearly holds in this case as well. Indeed, since $\zeta$ is a M\"obius transformation the complex derivative $\zeta'$ (hence also partial derivatives) does not vanish in $\mathbb{C}$. Also $\zeta(-\frac{1+w_0}{2})=0$ and $\zeta(z)>0$ for $z\in B$, so $\partial_1\zeta$ is not negative at $-\frac{1+w_0}{2}$. Therefore we have the
$\zeta(z)\simeq \dist(z,-\frac{1+w_0}{2})$ as $z$ approaches $-\frac{1+w_0}{2}$ along the real axis.

Just like before, for $x\in X_0$ there is an isometric copy $B_x$ of the region $B$ (i.e. $B_x = s_x(B)$ where $s_x$ is an isometry of the plane), such that $B_x\cap \tilde{T}\neq\emptyset$ and $B_x$ is tangent to $\tilde{T}$ at its special point $s(-\frac{1+w_0}{2})=x\in X_0$. Next, defining $\zeta_x$ on $B_x$ by $\zeta_x=\zeta\circ s_x^{-1}$, we obtain that the boundary circle where $\zeta_x=1$ lies outside $\tilde{T}$ on the other side of $X_1$ (the boundary component of $\tilde{T}$ where $v=1$).
As above, it is easy to check that $\zeta_x \leq v$ on the intersection  $B_x\cap\tilde{T}$ and in particular, this is true for $z$ which are on the normal segment to $X_0$ at $x$, which proves completely (\ref{est:auxiliary}).

Finally, the estimate $\zeta_x(z)\simeq  \dist(z,X_0)$ is obtained from the corresponding estimate for $\zeta(z)$ the same way as (\ref{est:auxiliary2}) followed for $\xi_x$.

Thus we proved (\ref{est:auxiliary}),(\ref{est:auxiliary2}),(\ref{est:auxiliary3}) which in turn imply (\ref{imaginary-estimate1}). The argument for (\ref{imaginary-estimate2}) that $1-v(z) \simeq \dist(z,X_1)$ is identical to the one above, so this proves the lemma.
\end{proof}
As mentioned before, this completes the proof of Theorem \ref{first thm}.
\end{proof}

\begin{remark}
The rounding  of the tube $T$ in the proof is not actually necessary. If we
leave the corners of $T$ then the derivative of the conformal map of $R$ to
$T$ will tend to zero or $\infty$ at the corners, but will have uniform
bounds on any subregion that is bounded away from the corners. The next
generation of the construction will take place inside such a sub-region, and
so the proof of the theorem would  work even without rounding the corners
(however, rounding is easy and gives a cleaner estimate).
\end{remark}

\begin{remark}
\label{reductionremark}
The case when $t$ is irrational follows from the rational case considered above.
Indeed, suppose $0<t<1$ and $\eps>0$ are like in the statement of Theorem \ref{first thm}, and $t$ is irrational. Then
we may choose $t<t'<1$ and $\eps'>0$ with $t'\in\mathbb{Q}$ so that
$$\frac{1-\eps'}{t'+1}\geq\frac{1-\eps}{t+1},$$
e.g. if $\eps$ is small enough we may take $t'=t+\eps(1+t)$ and $\eps'\leq\eps^2$. By the case considered above there is an Ahlfors-$t'$ regular Cantor set $S'\subset\mathbb{R}$, such that for every Borel subset $A$ of $\mathbb{R}\times S'$ we have
$$\frac{\dim_H f(A)}{\dim_H A} \geq \frac{2(1-\eps')}{t'+1} \geq \frac{2(1-\eps)}{t+1}.$$
Now, by a theorem of Mattila and Saaranen \cite{MattilaSaaranen} there is an Ahlfors-$t$ regular subset $S\subset S'$, since $t<t'$. Thus, considering Borel subsets $A$ of $\mathbb{R}\times S\subset\mathbb{R}\times S'$ completes the proof for the irrational values of $t$.
\end{remark}

\begin{proof}[Proof of Corollary \ref{cor:sharp-dimension}]
Let $d'$ and $\eta$ be as in the statement of the corollary.  Without loss of generality assume that $1<d'<2$. Let $d=\frac{2}{d'}-1-\eta$.  Then $\frac{2}{d+1}=\frac{d'}{1-\frac{\eta
d'}{2}}>d'$, and we may therefore choose $\epsilon>0$ so that
$\frac{2}{d+1}-\epsilon>d'$.

Theorem \ref{first thm} implies that for every $0<d<1$ there is a $d$-regular
Cantor set $S\subseteq \mathbb R$ and a quasiconformal mapping $f$ of
$\mathbb{R}^2$ satisfying inequality \eqref{first estimate}.  That is, for
every $y\in S$,
\[
\dim_H f(I\times \{y\})\geq \frac{2}{d+1} -\epsilon>d'\text{,}
\]
so that $ \{y\in \mathbb{R}: \dim_H f(I\times \{y\})> d'\}\supseteq S\text{.}
$ Since $\dim_H S = d$, it follows that
\[
    \dim_H\{y\in \mathbb{R}: \dim_H f(I\times \{y\})> d'\} \geq \dim_H(S)=d = \frac{2}{d'}-1-\eta\text{.}
\]
\end{proof}

\section{Proof of Theorem \ref{thm:unrectifiable}} \label{sec:unrectifiable}
\begin{proof}[Proof of Theorem \ref{thm:unrectifiable}]

The idea is quite simple. We start  by mapping horizontal tubes to nearly
horizontal tubes that oscillate. Inside these, we build thinner tubes that
oscillate on a smaller scale, and continue by induction,
 obtaining in the limit a Cantor set of curves  each
of which oscillates at infinitely many scales, hence has no rectifiable
subarc. Using sufficiently many sufficiently thin tubes we can keep the
quasiconstant bounded while making  the Hausdorff measure as large as we
want.

 The proof is essentially a sequence of pictures.
First, choose a diffeomorphism $\phi$ of the unit square $Q=[0,1]^2$
 to itself of the form $\phi(x,y)=(x,\psi(x,y))$ that is the identity on the boundary, and  translates the vertical segment
 $V=\{\frac 12\} \times [\frac 14, \frac 34]$
 up by $1/8$ to the segment
$\{\frac 12\} \times [\frac  38, \frac 78]$. Thus  segments of the form $S_y
= [0,1]\times \{ y\} \subset Q$ have curved images with the same endpoints as
$S_y$, but deviate by at least $ \frac 18$ from $S_y$ for
$y\in[\frac{1}{4},\frac{3}{4}]$. See Figure \ref{10}.

\begin{figure}[htb]
\centerline{
\psfig{figure=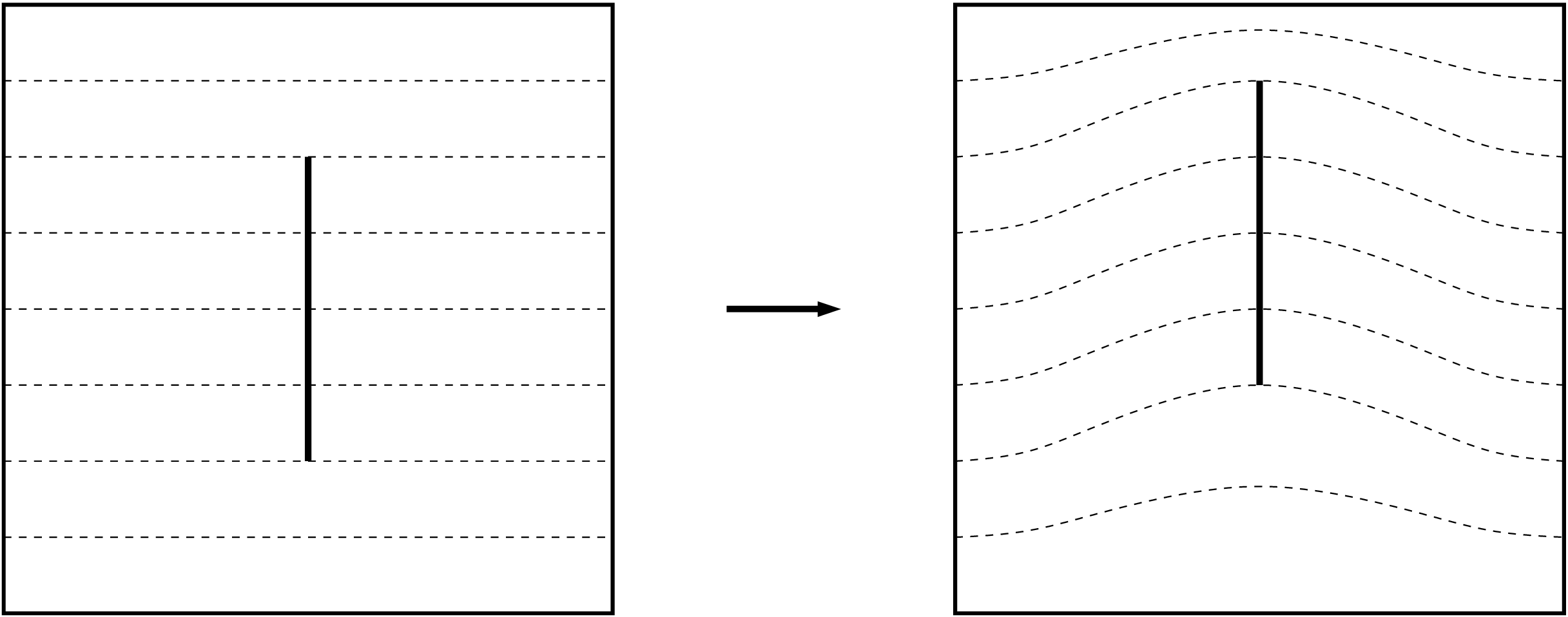,height=1.5in}
}
\caption{ \label{10}
Choose a smooth self-map of $Q$  that is the identity
on $\partial Q$, but causes horizontal segments near
the middle of $Q$ to bend by a definite amount.
}
\end{figure}

Let $\gamma =  \phi(S_y)$ for $y\in [\frac 14, \frac 34]$. Choose a large
integer $n$ and divide  $\gamma$ into  subcurves with endpoints $z_k = (x_k,
y_k)$
 where $x=\frac kn$, $k=0, 1, \dots n$.
Let $\gamma_n$ be the polygonal path with these vertices.  At each $z_k$,
$k=1, \dots n-1$, draw a segment perpendicular to and above $\gamma$ whose
length is the  same as the $(k-1)$st segment in $\gamma_n$.  The other
endpoint is  denoted $w_k$. Let $\tilde \gamma$ be the polygonal curve with
vertices $w_0, \dots, w_n$. See Figure \ref{9}.

\begin{figure}[htb]
\centerline{
\psfig{figure=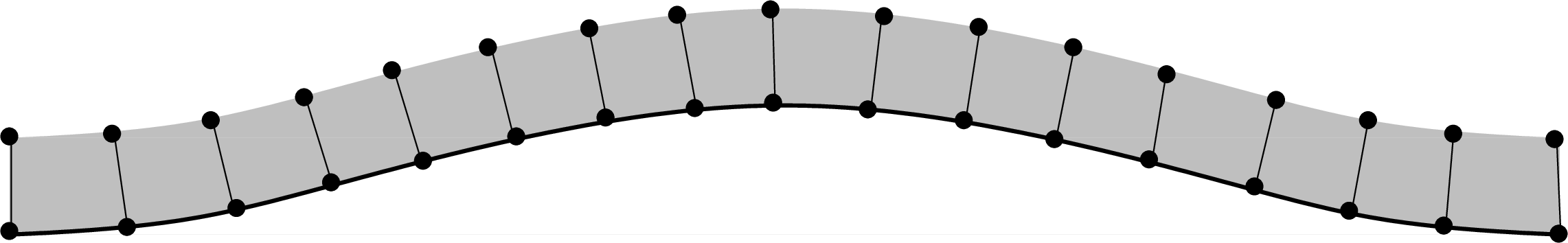,height=.75in}
}
\caption{\label{9}
We inscribe a polygonal curve $\gamma_n$ in $\gamma$ and
add perpendicular segments. We then connect the new vertices
to form a path $\tilde \gamma_n$ that is almost parallel
to $\gamma_n$.}
\end{figure}

The reason we defined $\tilde \gamma_n$ as we did (and did not simply
translate $\gamma_n$ upwards), is so that the region between the curves can
be divided into quadrilaterals that are very nearly squares. We will denote
this region $T_n$ and call it a ``tube''. Since $\gamma$ is a smooth curve,
adjacent segments of $\gamma_n$ have angles that agree to within $O(1/n)$;
thus adjacent perpendicular segments have angles that agree to within
$O(1/n)$. Thus each of the quadrilaterals formed by $\gamma_n, \tilde
\gamma_n$ and the segments joining their vertices have all angles within
$O(1/n)$ of $90$ degrees. See Figure \ref{8}.

\begin{figure}[htb]
\centerline{
\psfig{figure=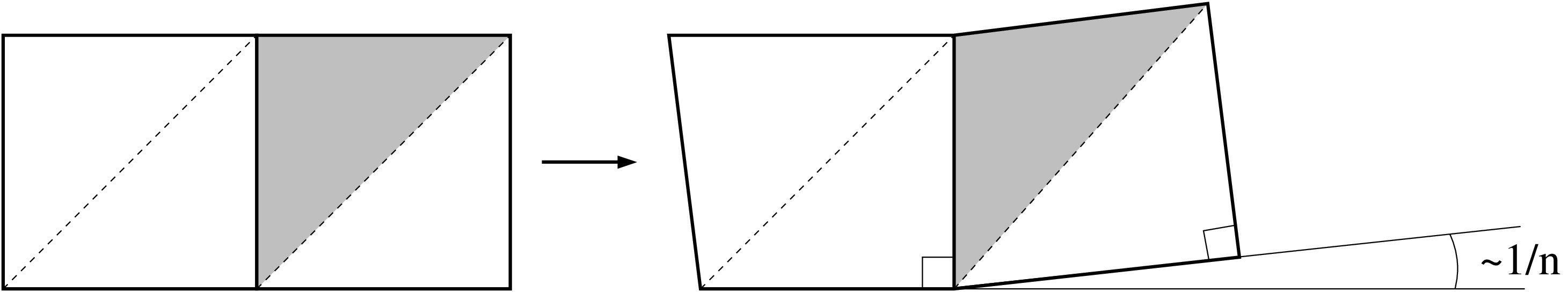,height=.8in}
}
\caption{ \label{8}
Two adjacent squares are mapped to adjacent  quadrilaterals
that are almost squares (with error $O(1/n)$).
}
\end{figure}

By adding diagonals and using the unique piecewise linear map between
corresponding triangles, we get a quasiconformal map from a true square to
each of our quadrilaterals with dilatation bounded by $O(1/n)$. Piecing these
together gives us a map from a $1 \times  \frac 1n$ rectangle to our  tube
$T_n$. See Figure \ref{9B}.

\begin{figure}[htb]
\centerline{
\psfig{figure=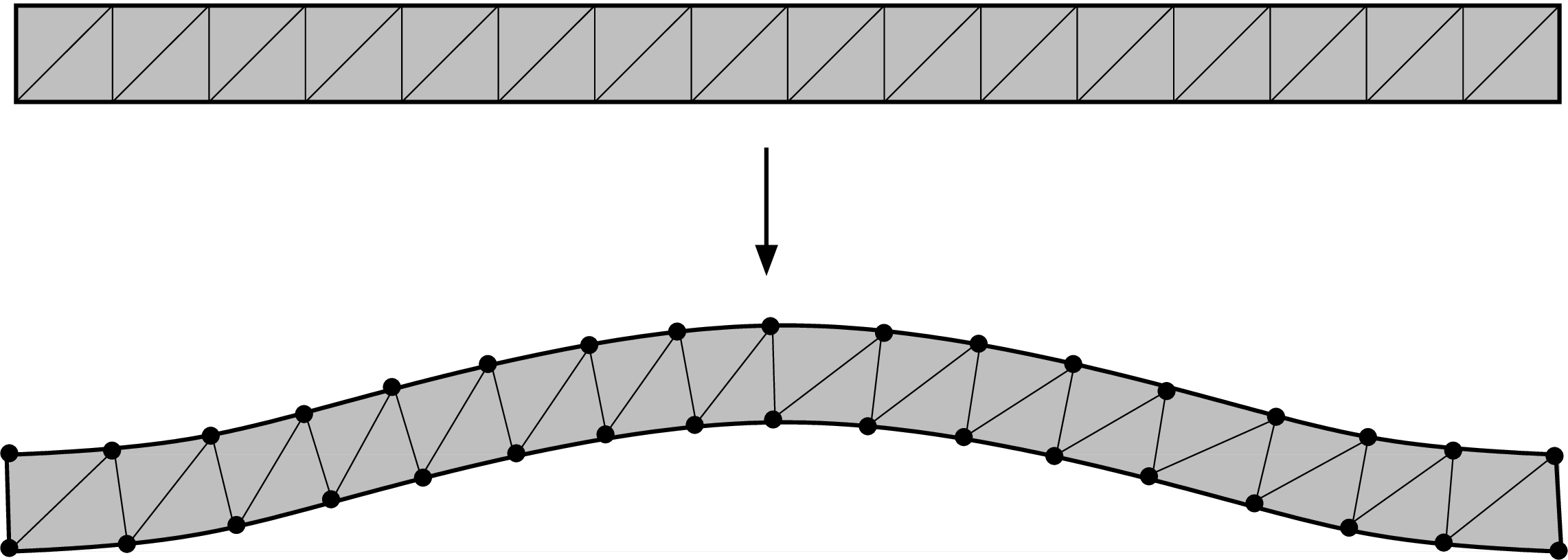,height=1.5in}
}
\caption{ \label{9B}
Piecewise linear maps define a $1+O(\frac 1n)$-QC map
from a rectangle to a tube.
}
\end{figure}

By repeating the  construction we can map several parallel straight tubes to
several  curved tubes as in Figure \ref{6}. We assume  there are $\simeq n$
 tubes, all  have width
$\simeq 1/n$ and are vertically separated by $\simeq 1/n$. This, combined
with the fact that the edges of the tubes have bounded slope, implies that
the piecewise linear map previously defined on the union of tubes can be
extended to a quasiconformal map  $\tilde{f}_n$ of $Q$ to itself with
uniformly bounded constant $K$ (independent of $n$).

\begin{figure}[htb]
\centerline{
\psfig{figure=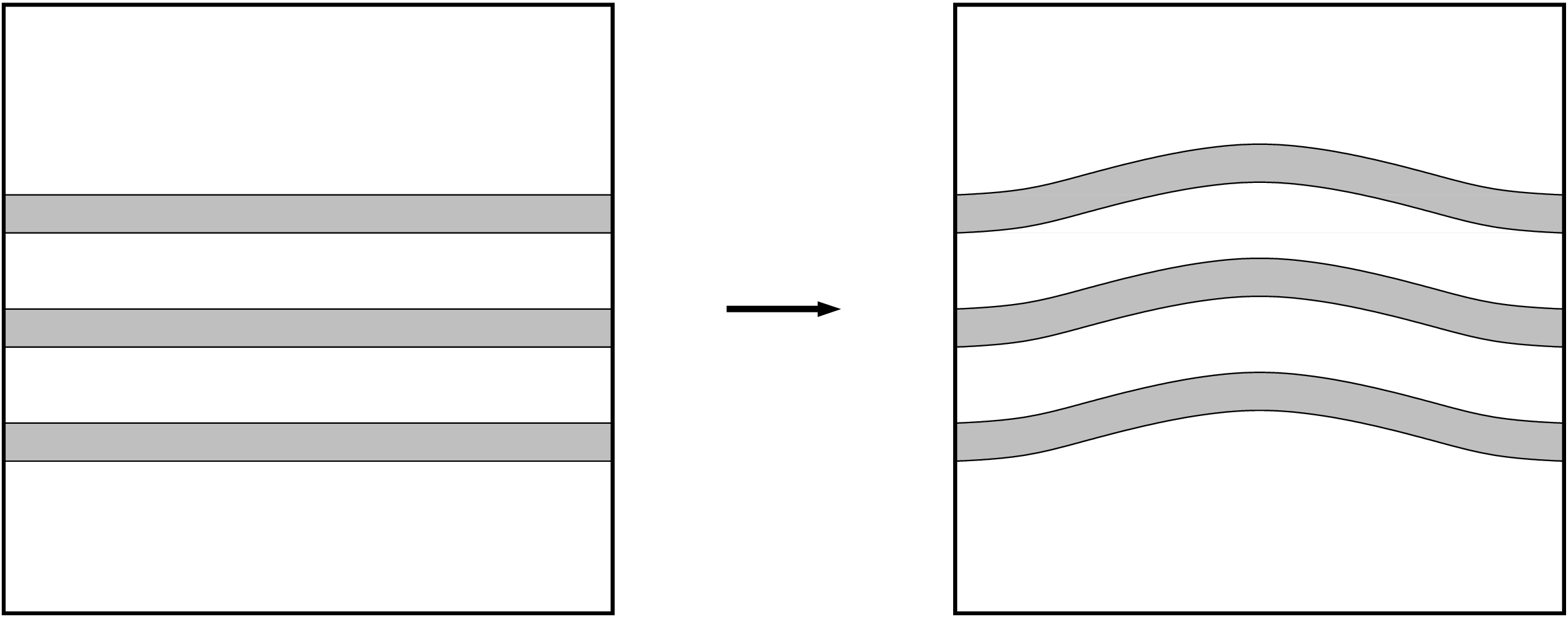,height=1.5in}
}
\caption{ \label{6}
Repeat the construction on several parallel tubes
}
\end{figure}

Now  we  must repeat the construction  at smaller scales, without letting the
dilatations blow up. However, this is quite simple.  We construct a Cantor
set as before, with nested families ${\cal I}_1 \supset {\cal I}_2 \supset
\dots$ of intervals.

For any $I \in {\cal I}_k$, define  a subcollection of  $\simeq n_{k+1} $
sub-intervals in the middle half of $I$ that have length $|I|/n_{k+1}$ and
are separated from each other by at least   $ |I|/n_{k+1}$. Doing this for
each $I\in {\cal I}_k$ defines ${\cal I}_{k+1}$.

Now, let ${\cal R}_k=\{[0,1]\times I:I\in{\cal I}_k\}$ be the family of $k$th
generation rectangles.  Each of these may be divided into squares, on which
we apply a scaled copy of the map
$\tilde{f}_{n_{k+1}}$, to obtain a $K$-quasiconformal map of $Q$ to itself,
which restricts to the identity on $Q\backslash \bigcup {\cal R}_k$, and is
$1+O(\frac{1}{n_{k+1}})$-QC on $\bigcup {\cal R}_{k+1}$.  Denote this map by
$f_{n_k}$.

Now, let $g_k = f_{n_1}\circ\dotsb\circ f_{n_k}$.  Then $g_k$ is $K+ O( \frac 1{n_1} + \dots + \frac 1{n_k})$ quasiconformal, as every point in $\bigcup {\cal R}_j\backslash \bigcup {\cal R}_{j+1}$ is hit by some number of iterations of the identity, then by a $K$-QC map, and then by successive $1+\frac{1}{n_i}$ quasiconformal maps, $i< j$. 

Thus if we choose $n_k$ to increase so quickly that $\sum_k n_k^{-1} $
converges to a small value, the dilatations will be uniformly bounded. In
this case, we can take a limit and obtain a map and a Cantor set $S$ so that
every segment $[0,1]\times\{y\}$, $y \in S$ has a definite oscillation  near
every point at infinitely many scales, and hence is nowhere rectifiable.

The set $S$ comes with a covering by the  nested collections of intervals
${\cal I}_k$. The Hausdorff measure  of $S$ can be bounded from below in the
usual way of defining a measure on $S$ by distributing mass from one of these
intervals to all its children equally.   Since   $\limsup_{n\rightarrow
\infty} n h( \delta/  n) =\infty$ for any fixed $\delta >0$, by choosing
$n_k$ large enough we can insure that
$$ \sum h(|I_j|)  \geq 2 h(|I|),$$
where the sum is over the  ${\cal I}_{k+1}$  of $I \in {\cal I}_k$. A
standard argument then shows that $S$ has infinite $h$-measure.
\end{proof}



\section{Remarks and Open Questions}\label{Section:remarks&corollaries}

Many natural questions in the vein of our results still remain open. The
following question was suggested to us by Nages Shanmugalingam.

\begin{question}
Let $f:X\to Y$ be a quasisymmetric mapping between Ahlfors $D$-regular
($D$-Loewner) spaces, $D>1$. Is it true that for a fixed $d\in[0,D]$ and for $\m_{D/d}$
almost every Ahlfors $d$-regular set $E\subset X$, we have that $f(E)$ is
Ahlfors $d$-regular as well?
\end{question}

This question is open even in $\mathbb{R}^N$, except for $d=N$, which follows
from a theorem of Gehring and Kelly \cite{GehringKelly}. The case of $d=D$
for more general metric spaces was considered by Korte, Marola and
Shanmugalingam \cite{KMShan}.




\subsection{Families of translates}
The questions below are formulated in the case of the Euclidean space, but
would also be interesting for Carnot groups.

How large is the collection of translates of a set $E$ the dimension of which
can simultaneously jump up by a prespecified amount? To be more precise, let
$E\subset \mathbb{R}^N$ be a bounded Ahlfors $d$-regular set, $d\in(0,N)$,
and $f$ a quasiconformal map of $\mathbb{R}^N$. For $d'\in(d,N)$, we would
like to estimate from above the Hausdorff dimension of the points $y$ s.t.
$\dim_H f(y + E)
> d'$?
%

\begin{question}\label{question:translates}
Suppose $E\subset\mathbb{R}^N$ is a bounded $d$-regular set. Is it true that
for $d'\in(d,N)$ and for every QC map $f$ of $\mathbb{R}^N$ we have
 \begin{equation}\label{question:dimincrease}
\dim_H \{y\in \mathbb{R}^N : \dim_H f(y + E) > d'\} \leq
      \frac{d}{d'}N?
 \end{equation}
\end{question}
%
%
%

To estimate the dimension of the set appearing in (\ref{question:dimincrease}) we
would like to use Corollary \ref{cor:expansion} and a stronger version of
Lemma \ref{lem:translatesmodulus}, which we formulate in the case of
$\mathbb{R}^N$.

Recall from Section \ref{Section:translatesproductsproofs}, that if
$y\in\mathbb{R}^N$ and $\la$ is a measure on $\mathbb{R}^N$ we denoted by
$y_*\la$ the pushforward of $\la$ by the translation by $y$, and for
$K\subset\mathbb{R}^N$ we let  $K_*\la=\{y_*\la : y\in K\}$.

\begin{question}\label{question:modulus}
Let $p\geq 1$, $1<D < N$ and $K\subset\mathbb{R}^N$. Suppose that
$\mod_p(K_*\la,\mu)=0,$ whenever $\mu$ is an upper $D$-regular measure on
$\mathbb{R}^N$. Is $\dim_H K \leq D$?
\end{question}

By Corollary \ref{cor:expansion}, an affirmative answer to Question
\ref{question:modulus} would imply an affirmative answer to Question
\ref{question:translates}.

\subsection{Sharpness for nonplanar mappings} Though Theorems \ref{thm:nonexpand} and \ref{expandsmetricquantitative} are quite general, our constructions in Theorems \ref{first thm} and \ref{thm:unrectifiable} seem difficult to extend beyond the planar case, as we do not have flexibility in constructing conformal mappings in higher dimensions.  This leads to the following questions:
\begin{question}
Does an analogue to Theorem \ref{first thm} remain true in higher dimensions?
For example, given positive integers $m,n$, $N=m+n\geq 3$, and $t\in (0,n)$,
is there a  $t$-regular subset $S\subseteq \mathbb R^n$ such that for each
$\epsilon>0$, there is a quasiconformal map $f\colon \mathbb R^N\rightarrow
\mathbb R^N$, such that for every Borel set $A\subseteq \mathbb R^m\times S$,
\begin{equation*}
\dim_H(f(A))\geq (1-\epsilon)\frac{N\dim_H(A)}{t+m}\text{?}
\end{equation*}
\end{question}
\begin{question}
Suppose $m$, $n$, $N$ are as before, and that
\begin{equation}
\label{gaugelimithigherdim}
\lim_{t \to 0}    \frac {h(t)}{t^n} = \infty\text{,}
\end{equation}
Is there a compact set $S\subseteq [0,1]^{n}$ and a quasiconformal map
$f\colon \mathbb R^N\rightarrow \mathbb R^N$ so that
\begin{enumerate}
\item The quasiconformal constant of $f$ is bounded independent of $h$ and
    $S$.
\item $S$ has infinite Hausdorff measure with respect to $h$ (cf.\
    Definition \ref{def:hmeas} below).
\item $f([0,1]^m\times \{y\})$ contains no $m$-rectifiable subset for any
    $y \in S$.
\end{enumerate}
Recall that a set is $m$-rectifiable if it is the union of an
$\calH^n$-nullset and a countable family of Lipschitz images of subsets of
$\mathbb R^m$.
\end{question}

In the preceding question, one can consider other conditions in place of the
absence of rectifiable subsets. For example, one could still ask in higher
dimensions that the surfaces $f([0,1]^m\times \{y\})$ contain no rectifiable
curves.

An even stronger requirement would be that the surfaces can be parameterized
by a map satisifying $|\phi(x) - \phi(y)| > \eta(|x-y|)$ for some $\eta$ with
$\eta(t)/t \to \infty$ as $t \to 0$".  It is easy to see from the proof that
our planar example in  Theorem \ref{first thm} has this latter property.

In higher dimensions, examples of such highly non-rectifiable surfaces were
constructed by David and Toro in  \cite{DavidToro} and by the first author in
\cite{Bishop:surface}.

\subsection{Distortion by Sobolev mappings} We lastly point out that these questions may be, and have been, explored in other classes of mappings.

As mentioned before, frequency of dimension distortion of Sobolev mappings
$f\in W^{1,p}(\mathbb{R}^n, Y)$ were studied in \cite{Tyson:frequency} for
$p\geq n$. Dimension distortion in the case $1<p<n$ has been explored by
Hencl and Honz\'{i}k, cf. \cite{Hencl:subcritical},
\cite{Hencl:dimdistortion}. Similar questions for more general source spaces,
e.g. for the Heisenberg group, have been considered by Balogh, Mattila, Tyson
and Wildrick in \cite{BMT:Grassmanian}, \cite{BTW:dimdistortion}
and\cite{BTW:heisenberg}. Hencl and Honz\'{i}k also considered dimension
distortion for  mappings in Triebel-Lizorkin spaces in
\cite{Hencl:triebel-lizorkin}.

As for results in the vein of Theorems \ref{first thm} and
\ref{thm:unrectifiable}, examples of mappings which simultaneously expand
large families of subspaces or subsets have been exhibited in
\cite{Tyson:frequency}, \cite{BTW:dimdistortion}, \cite{Hencl:subcritical}
and \cite{Hencl:triebel-lizorkin} in various contexts. In fact it was shown
in these works that such mappings are ``generic" in some sense. We refer the
reader to the mentioned papers for more precise statements and definitions.

Even so, it is unclear to what extent our results extend to the Sobolev case.
For example, we do not know the answer to the following question.
\begin{question}
Suppose $N\geq2$ and $0<d<N$. Let $f:\mathbb{R}^N\to\mathbb{R}^N$ be a
continuous Sobolev mapping, $f\in W^{1,p}(\mathbb{R}^N,\mathbb{R}^N), p>N$.
Is it true that for $D/d$-almost every Ahlfors $d$-regular set
$E\subset\mathbb{R}^N$ we have $\dim_H f(E) \leq d$?
\end{question}

\end{document}